\documentclass[11pt]{article}

\usepackage{authblk}
\usepackage{fullpage}  
\usepackage{float}
\usepackage{pb-diagram}  		
\usepackage{esvect}
\usepackage[utf8]{inputenc} 
\usepackage[T1]{fontenc}    
\usepackage{hyperref}       
\usepackage{url}            
\usepackage{booktabs}       
\usepackage{amsfonts}       
\usepackage{nicefrac}       
\usepackage{microtype}      
\usepackage{fullpage}  
\usepackage{float}
\usepackage{pb-diagram}  
\usepackage{makecell}		

\usepackage{url}

\usepackage{algorithm, algorithmicx,algpseudocode, listings} 
\usepackage{multirow} 
\usepackage{hhline}  
\usepackage{amsmath}
\usepackage{xcolor}

\usepackage{graphicx}
\usepackage{grffile}
\usepackage{amscd}
\usepackage{amssymb,mathrsfs}
\input{epsf.sty}
\usepackage{amsthm,amscd}
\usepackage{color}
\usepackage{latexsym}
\usepackage{epic}
\usepackage{appendix}
\usepackage{enumerate}
\usepackage{longtable}
\usepackage{lscape}
\usepackage{extarrows}
\usepackage{epstopdf}
\usepackage{listings}
\usepackage{subfigure}
\newlength\myindent

\newcommand{\Rmnum}[1]{\expandafter\@slowromancap\romannumeral #1@}
\newcommand{\inner}[3][]{{\left\langle #2,#3 \right\rangle_{#1}}}

\allowdisplaybreaks

\newcommand\smallO{
  \mathchoice
    {{\scriptstyle\mathcal{O}}} 
    {{\scriptstyle\mathcal{O}}} 
    {{\scriptscriptstyle\mathcal{O}}}
    {\scalebox{.6}{$\scriptscriptstyle\mathcal{O}$}}
  }

 \newcommand{\whcomm}[2]{{\sf\color{purple} #1}{\sf\color{blue} #2}}

\newtheorem{definition}{Definition}[section]
\newtheorem{theorem}{Theorem}[section]
\newtheorem{lemma}{Lemma}[section]
\newtheorem{proposition}{Proposition}[section]
\newtheorem{corollary}{Corollary}[section]
\newtheorem{assumption}{Assumption}[section]
\newtheorem{remark}{Remark}[section]
\newtheorem{example}{Example}[section]
\numberwithin{equation}{section}

\DeclareMathOperator{\prox}{\mathrm{prox}}

\DeclareMathOperator{\F}{\mathrm{F}}
\DeclareMathOperator{\T}{\mathrm{T}}
\DeclareMathOperator{\Hess}{\mathrm{Hess}}
\DeclareMathOperator{\grad}{\mathrm{grad}}

\DeclareMathOperator{\Exp}{\mathrm{Exp}}

\DeclareMathOperator{\I}{\mathrm{I}}
\DeclareMathOperator{\N}{\mathrm{N}}
\DeclareMathOperator{\D}{\mathrm{D}}

\DeclareMathOperator{\trace}{\mathrm{trace}}

\DeclareMathOperator{\sign}{\mathrm{sign}}

\DeclareMathOperator{\St}{\mathrm{St}}

\DeclareMathOperator*{\argmin}{arg\,min}

\def\R{\mathbb{R}}
\def\M{\mathcal{M}}
\begin{document}

\title{A Riemannian Proximal Newton-CG Method 
\footnotetext{Corresponding author: wen huang (\url{wen.huang@xmu.edu.cn}). WH was partially supported by the National Natural Science Foundation of China (No. 12001455 and No. 12371311), the National Natural Science Foundation of Fujian Province (No. 2023J06004), and Xiaomi youth scholar funding (2024). Wutao Si was partially supported by the National Natural Science Foundation of China (No. 12171403).}}
\author[1]{Wen Huang }
\author[1]{Wutao Si}

\affil[1]{ School of Mathematical Sciences, Xiamen University, Xiamen, China.\vspace{.15cm}}

\maketitle


\begin{abstract}
Recently, a Riemannian proximal Newton method has been developed for optimizing problems in the form of $\min_{x\in\mathcal{M}} f(x) + \mu \|x\|_1$, where $\mathcal{M}$ is a compact embedded submanifold and $f(x)$ is smooth. Although this method converges superlinearly locally, global convergence is not guaranteed. The existing remedy relies on a hybrid approach: running a Riemannian proximal gradient method until the iterate is sufficiently accurate and switching to the Riemannian proximal Newton method. This existing approach is sensitive to the switching parameter. This paper proposes a Riemannian proximal Newton-CG method that merges the truncated conjugate gradient method with the Riemannian proximal Newton method. The global convergence and local superlinear convergence are proven. Numerical experiments show that the proposed method outperforms other state-of-the-art methods.
\end{abstract}

\section{Introduction}

Many important applications can be formulated as minimizing a composite function on a Riemannian manifold:
\begin{equation} \label{eq:54}
\min_{x \in \mathcal{M}} F(x) = f(x) + h(x),
\end{equation}
where $\mathcal{M}$ is a Riemannian manifold, $f$ is smooth, and $h$ is continuous but may be nonsmooth. Such applications include but are not limited to sparse principal component analysis \cite{ZHT2006,ZX2018}, sparse partial least squares regression~\cite{CSGHY2018b}, compressed mode~\cite{OLCO2013}, clustering~\cite{HWGV2022,LYL2016,PZ2018}, image inpainting~\cite{PH2023}, and face expression recognition~\cite{FRJA2018}.

Recently, much attention has been paid to developing algorithms for solving Problem~\eqref{eq:54}. Many algorithms are inspired by the Euclidean proximal gradient-type methods~\cite{Beck2017,LSS2014}. In~\cite{CMSZ2019}, a proximal gradient method, named ManPG, is developed for the manifold $\mathcal{M}$ being the Stiefel manifold $\St(r, n) = \{X \in \mathbb{R}^{n \times r} \mid X^T X = I_r\}$. ManPG approximates the objective function $F$ around an iterate $x$ by a first-order approximation of $f$, i.e.,
$F(x+\eta) \approx f(x) + \inner[]{\nabla f(x)}{\eta} + \frac{L}{2} \|\eta\|_{\F}^2 + h(x + \eta)$, where $\eta \in \T_x \mathcal{M}$, $\T_x \mathcal{M}$ denotes the tangent space of $\mathcal{M}$ at $x$, $L > 0$ is a constant, and $\inner[]{}{}$ denotes the Euclidean inner product.
It follows that the update of ManPG is given by
\begin{align}
\eta_k =& \argmin_{\eta \in \T_{x_k} \mathcal{M}} \inner[]{\nabla f(x_k)}{\eta} + \frac{L}{2} \| \eta \|_{\F}^2 + h(x + \eta) \label{eq:55} \\
x_{k+1} =& R_{x_k}(\alpha_k \eta_k), \nonumber
\end{align}
with $\alpha_k$ being an appropriate step size, where $R$ is a retraction and plays the same role as the addition in the Euclidean space, see more details in Section~\ref{sec08}. It is proven in~\cite{CMSZ2019} that ManPG converges globally. Subsequently, a different version of the proximal gradient method, called RPG, is proposed in~\cite{HuaWei2019b}. The addition $x+\eta$ in~\eqref{eq:55}, well-defined only when the manifold has a linear ambient space, is replaced by $R_x(\eta)$. Thus, RPG allows the manifold to be generic, and convergence rates under Riemannian convexity and Riemannian KL property are also established. Despite the nice theoretical results, the RPG subproblem is generally difficult to solve, unlike ManPG where the subproblem in~\eqref{eq:55} can be solved efficiently by a semismooth Newton method. In~\cite{WY2023}, the idea of quasi-Newton update is further integrated with ManPG, and a proximal quasi-Newton method on manifolds, named ManPQN, is proposed. The second-order term $\frac{L}{2}\|\eta\|_{\F}^2$ in~\eqref{eq:55} is replaced by a weighted norm $\frac{1}{2}\|\eta\|_{B}^2 = \frac{1}{2} \inner[]{\eta}{B \eta}$, where the matrix $B$ is the diagonal entries in the LBFGS update. The local linear convergence rate is proven therein. The recent paper~\cite{SAHJV2024} shows that even though the search direction $\eta_k$ in ManPQN is computed using a second-order approximation of $f$, the local convergence rate is still linear at most generally since $x+\eta$ in~\eqref{eq:55} is only a first-order approximation of a nonlinear manifold. In~\cite{WY2024}, this difficulty is overcome by using $h(R_x(\eta))$ rather than $h(x+\eta)$, analogous to RPG, and an adaptive regularized proximal Newton algorithm (ARPN) is proposed. ARPN has desired theoretical results in the sense that it converges globally and has a local superlinear convergence rate under certain assumptions. However, the subproblem of ARPN is even more difficult to solve compared to RPG and no numerical experiments are given regarding ARPN in~\cite{WY2024}.
In~\cite{SAHJV2024}, a different approach is used to develop a Riemannian proximal Newton method. It is shown that the search direction $\eta_k$ of ManPG in~\eqref{eq:55} is a semismooth function of $x_k$ and a semismooth Newton method can be used to find a root of $\eta_k$. It follows that a Remannian proximal Newton method (RPN) is proposed and a local superlinear convergence rate is guaranteed. The local convergence rate is later improved to be quadratic in~\cite{SAHJV2024arxiv}. RPN is practical compared to ARPN since its subproblem involves computing~\eqref{eq:55} and solving a linear system, which can be solved efficiently.

Though RPN has a fast local convergence rate, the globalization of RPN is still problematic. In~\cite{SAHJV2024}, the authors suggest first using ManPG until the iterate is sufficiently close to a minimizer and then switching to RPN to ensure fast local convergence. However, the switching time is difficult to choose and a bad choice can influence the performance of the overall algorithm. In this paper, we integrate the truncated conjugate gradient method in~\cite{DS83NewtontCG} with RPN and propose a Riemannian proximal Newton-CG method, which guarantees global convergence and local superlinear convergence rate. Specifically, we consider the same optimization problems as those in~\cite{SAHJV2024}, that is, in the form of
\begin{equation} \label{eq:F}
\min_{x \in \mathcal{M}} F(x) = f(x) + h(x) \hbox{ with } h(x) = \mu \|x\|_1,
\end{equation}
where $\mathcal{M}$ is a $d$-dimensional compact embedded submanifold of a $n$-dimensional Euclidean space 
and the function $f$ is smooth. The Riemannian proximal Newton-CG method consists of three phases:
\begin{itemize}
    \item the first phase computes the ManPG proximal gradient direction as in~\eqref{eq:55};
    \item the second phase uses a truncated conjugate gradient to approximately solve the semismooth Newton equation; and
    \item the last phase updates the iterates by a retraction with appropriate step sizes.
\end{itemize}
Due to the nonlinearity of the manifold, the developments of the second and third phases are not trivial since either the difficulty does not exist in the Euclidean setting or simple generalizations of Euclidean versions do not work as expected. The contributions of this paper are summarized as follows:
\begin{itemize}
    \item The semismooth Newton equation in RPN, which does not have a symmetric coefficient matrix, is reformulated as a quadratic optimization problem. Note that the size of the optimization problem equals the number of nonzero entries in the iterate $x$.
    \item The early termination conditions in the truncated conjugate gradient are developed such that a descent direction is guaranteed if the early termination conditions take effect. Note that those early termination conditions differ from the existing ones in~\cite{DS83NewtontCG,NocWri2006}.
    \item If the truncated conjugate gradient terminates due to the accurate condition, then its output is a superlinear search direction under certain reasonable assumptions. Moreover, the function value is sufficiently descent with step size one every two steps when iterates are sufficiently close to a minimizer. Note that the function value may not even decrease with step size one even if the search direction is a quadratic convergence direction due to the curvature of retraction, as shown in Example~\ref{ex01}. Note that such difficulty does not appear for the proximal Newton method for Euclidean nonsmooth problems or the Riemannian Newton method for Riemannian smooth problems.
    \item The global convergence and local superlinear convergence are established.
    \item Extensive numerical experiments show that the proposed method performs as expected and is more efficient and effective than existing proximal gradient-type methods. 
\end{itemize}
There are also other types of algorithms for solving problems in the form of~\eqref{eq:54}, including an augmented Lagrangian method~\cite{ZBDZ2022}, and a Riemannian ADMM~\cite{LMS2022}. In~\cite{ZBDZ2022}, the global convergence of the augmented Lagrangian method is given and the local linear convergence rate is established later in~\cite{ZBD2022}. In~\cite{LMS2022}, a Riemannian ADMM is proposed with global convergence analysis. These two methods do not have local superlinear convergence results. Therefore, \emph{the proposed method is the first practical second-order algorithm that converges globally and locally superlinearly for nonsmooth optimization on manifolds.}

When the manifold $\mathcal{M}$ is the Euclidean space $\mathbb{R}^n$, the proposed Riemannian proximal Newton-CG method does not become any of the existing Euclidean algorithms, as far as we know. The closest approaches in the Euclidean setting first use the proximal gradient step to identify the underlying manifold structure and then optimize over the manifold structure to accelerate the algorithm, see~\cite{BIM2023,LFP2017}. Such existing approaches do not motivate the corresponding algorithms by the semismoothness of $\eta_k$ and finding a zero of $\eta_k$. The resulting algorithm is therefore different. See detailed comparisons in Remark~\ref{re01}.

This paper is organized as follows. Section~\ref{sec07} gives notation and preliminaries of Riemannian optimization and the existing Riemannian proximal Newton method. Section~\ref{sec09} states the proposed Riemannian proximal Newton-CG method. The global convergence and the local superlinear convergence rate are established therein. Numerical experiments are given in Section~\ref{sec12}. Finally, the conclusion is drawn in Section~\ref{sec13}.

\section{Notation and Preliminaries} \label{sec07}

The $n$-dimensional Euclidean space is denoted by $\mathbb{R}^n$. Note that $\mathbb{R}^n$ does not just refer to a vector space. It also refers to matrix space $\mathbb{R}^{s \times t}$ with $n = s t$, tensor space $\mathbb{R}^{n_1 \times \ldots \times n_d}$ with $\prod_{i = 1}^d n_i = n$, and any $n$-dimensional linear space. The Euclidean metric on $\mathbb{R}^n$ is $\inner[]{u}{v} = \sum_{i = 1}^n u_i^T v_i = u^T v$ for vector space, $\inner[]{U}{V} = \sum_{i = 1}^{n_1} \sum_{j = 1}^{n_2} U_{ij} V_{ij} = \trace(U^T V)$ for matrix space, and $\inner[]{\mathbf{U}}{\mathbf{V}} = \sum_{(i_1, \ldots, i_d) = (1, \ldots, 1)}^{(n_1, \ldots, n_d)} \mathbf{U}_{i_1 \ldots i_d} \mathbf{V}_{i_1 \ldots i_d}$ for tensor space. For any $v \in \mathbb{R}^n$, the Frobenius norm is $\|v\|_{\F} = \sqrt{\inner[]{v}{v}}$ and the one norm $\|v\|_1$ is the sum of the absolute values of all entries in $v$. $\mathrm{sign}(v) \in \mathbb{R}^n$ denotes the sign function, i.e., $(\mathrm{sign}(v))_i = -1$ if $v_i < 0$; $(\mathrm{sign}(v))_i = 0$ if $v_i = 0$; and $(\mathrm{sign}(v))_i = 1$ otherwise. For any $u, v \in \mathbb{R}^n$, $u \odot v$ denotes the Hadamard product, i.e., $u \odot v = w \in \mathbb{R}^n$ such that $w_i = u_i v_i$ for all $i$.
Given a linear operator $\mathcal{A}:\mathbb{R}^n \rightarrow \mathbb{R}^m$, denoted by $\mathcal{A} \in \mathbb{R}^{m \times n}$, the operator norm of $\mathcal{A}$ is denoted by $\|\mathcal{A}\|_2$. The adjoint operator of $\mathcal{A}$ is denoted by $\mathcal{A}^{\sharp}$, satisfying $\inner[]{\mathcal{A} v}{u} = \inner[]{v}{\mathcal{A}^{\sharp} u}$ for all $v \in \mathbb{R}^n$ and $u \in \mathbb{R}^m$. If $\mathcal{A} = \mathcal{A}^{\sharp}$, then $\mathcal{A}$ is called a symmetric (self-adjoint) operator. The largest and smallest singular values of $\mathcal{A}$ are respectively denoted by $\sigma_{\max}(\mathcal{A})$ and $\sigma_{\min}(\mathcal{A})$. 
If $\mathcal{A}$ is symmetric, $m = n$, and $\inner[]{u}{\mathcal{A}u} > (\geq) 0$ holds for all $u \in \mathbb{R}^n$, then $\mathcal{A}$ is called a symmetric positive definite (semidefinite) operator, denoted by $\mathcal{A} \succ (\succeq) 0$. 
The largest and smallest eigenvalues of $\mathcal{A}$ are respectively denoted by $\chi_{\max}(\mathcal{A})$ and $\chi_{\min}(\mathcal{A})$ when $m = n$.
Let $I_n \in \mathbb{R}^{n \times n}$ denote the identity operator from $\mathbb{R}^n$ to $\mathbb{R}^n$ and let $0_{n \times m}$ denote the zero operator from $\mathbb{R}^m$ to $\mathbb{R}^n$. Given a linear subspace $\mathcal{L}$ of $\mathbb{R}^n$, the orthogonal projection from $\mathbb{R}^n$ to $\mathcal{L}$ is denoted by $P_{\mathcal{L}}$.

\subsection{Riemannian manifold} \label{sec08}

The concepts of Riemannian geometry and Riemannian optimization tools can be found in the standard literature, e.g.,~\cite{AMS2008,Boo1986,boumal2020intromanifolds}. The notation of this paper follows from~\cite{AMS2008}. In this paper, the manifold $\mathcal{M}$ refers to a finite-dimensional compact Riemannian embedded submanifold of $\mathbb{R}^n$. Roughly speaking, one can view a compact Riemannian embedded submanifold of $\mathbb{R}^n$ as a bounded, smooth surface in $\mathbb{R}^n$. 
We refer to~\cite{AMS2008,Lee2012} for more details. 
The tangent space of $\mathcal{M}$ at $x$ is denoted by $\T_x \mathcal{M}$ and the union of all tangent spaces, denoted by $\T \mathcal{M}$, is the tangent bundle of $\mathcal{M}$. The Riemannian metric of $\mathcal{M}$ at $x$, denoted by $\inner[x]{\cdot}{\cdot}$, defines an inner product in $\T_x \mathcal{M}$. The induced norm in $\T_x \mathcal{M}$ is denoted by $\|\cdot\|_x$.
Since $\mathcal{M}$ is a Riemannian embedded submanifold of $\mathbb{R}^n$, the tangent space $\T_x \mathcal{M}$ is a linear subspace of $\mathbb{R}^n$.
The Riemannian metric on $\mathcal{M}$ is endowed from its embedding Euclidean space and therefore is given by $\inner[x]{\eta_x}{\xi_x} = \inner[]{\eta_x}{\xi_x}$ for all $\eta_x, \xi_x \in \T_x \mathcal{M}$. The orthogonal complement space of $\T_x \mathcal{M}$, denoted by $\N_x \mathcal{M} = \T_x^{\perp} \mathcal{M}$, is the normal space of $\mathcal{M}$ at $x$. Let $B_x = \{b_1, b_2, \ldots, b_{n - d}\}$ denote an orthonormal basis of $\N_x \mathcal{M}$. Define linear operators $B_x: \mathbb{R}^{n - d} \rightarrow \N_x \mathcal{M}: v \mapsto \sum_{i = 1}^{n-d} v_i b_i$ and $B_x^T: \mathbb{R}^n \rightarrow \mathbb{R}^{n - d}: u \mapsto ( \inner[]{u}{b_1}, \inner[]{u}{b_2}, \ldots, \inner[]{u}{b_{n-d}})^T$. Therefore, the tangent space $\T_x \mathcal{M}$ can be characterized as $\T_x \mathcal{M} = \{ v \in \mathbb{R}^n \mid B_x^T v = 0\}$. By~\cite{HAG2016VT,SAHJV2024}, one can always choose $B_x$ to be a smooth mapping of $x$ in a sufficiently small neighborhood of $x$ for any $x \in \mathcal{M}$. Such a smooth mapping is useful in the local convergence analysis in Section~\ref{sec03}.

Given a smooth function $f$ on $\mathcal{M}$, the Riemannian gradient at $x$ is the unique tangent vector given by
$
\grad f(x) = P_{\T_x \mathcal{M}} (\nabla f(x)) \in \T_x \mathcal{M},
$
where $\nabla f(x)$ denotes the Euclidean gradient of $f$ at $x$. The Riemannian Hessian of $f$ at $x$ is a linear operator satisfying
\begin{align*}
\Hess f(x):\T_x \mathcal{M} \rightarrow \T_x \mathcal{M}: \eta_x \mapsto &\Hess f(x) [\eta_x] = P_{\T_x \mathcal{M}} (\D \grad f(x) [\eta_x] ) \\
=& P_{\T_x \mathcal{M}} ( \nabla^2 f(x) [\eta_x]) + \mathcal{W}_x^{\mathcal{M}}(\eta_x, P_{\N_x \mathcal{M}}(\nabla f(x)) ),
\end{align*}
where $\D \grad f(x) [\eta_x]$ denotes the direction derivative of $\grad f(x)$ along $\eta_x$, $\nabla^2 f(x)$ denotes the Euclidean Hessian of $f$ at $x$, and $\mathcal{W}_x^{\mathcal{M}}$ denotes the Weingarten map of $\mathcal{M}$ at $x$, i.e.,
$
\mathcal{W}_x^{\mathcal{M}}: \T_x \mathcal{M} \times \N_x \mathcal{M} \rightarrow \T_x \mathcal{M}:(w, u) = (\D (x \mapsto P_{\T_x \mathcal{M}})(x) [w]) u.
$
Note that the Weingarten map is a symmetric operator on $\T_x \mathcal{M}$ in the sense that given $u \in \N_x \mathcal{M}$, it holds that $\inner[x]{\xi_x}{\mathcal{W}_x^{\mathcal{M}}(\eta_x, u)} = \inner[x]{\eta_x}{\mathcal{W}_x^{\mathcal{M}}(\xi_x, u)}$ for all $\eta_x, \xi_x \in \T_x \mathcal{M}$.

The notion of retraction on a manifold is used to update the iterates of an algorithm and can be viewed as a generalization of addition in a Euclidean space, i.e., $x_{k+1} = R_{x_k}( \alpha_k \eta_{x_k})$ generalizes from $x_{k+1} = x_k + \alpha_k \eta_{x_k}$. Rigorously, a retraction is a smooth mapping from the tangent bundle $\T \mathcal{M}$ to $\mathcal{M}$ such that (i) $R_x(0_x) = x, \forall x\in \mathcal{M}$, where $0_x$ is the zero vector in $\T_x \mathcal{M}$; and (ii) the differential of $R_x$ at $0_x$ is the identity map, i.e., $\D R_x(0_x) = \mathrm{id}$. For any $x \in \mathcal{M}$, there exists a neighborhood of $x$ such that $R_x$ is a diffeomorphism in the neighborhood. Moreover, for a compact embedded submanifold $\mathcal{M}$, there exist two constants $C_{R_1}$ and $C_{R_2}$ such that
\begin{align}
\|R_x(\eta_x) - x\|_{\F} \leq C_{R_1} \|\eta_x\|_{\F} \label{eq:26} \\
\|R_x(\eta_x) - x - \eta_x\|_{\F} \leq C_{R_2} \|\eta_x\|_{\F}^2	\label{eq:27}
\end{align}
hold for all $x \in \mathcal{M}$ and $\eta_x \in \T_x \mathcal{M}$, see e.g.,~\cite{CMSZ2019}. One important retraction is the exponential mapping, denoted by $\Exp$, which is defined by $\Exp_x(\eta_x) = \gamma(1)$, $\gamma:I \rightarrow \mathcal{M}$ is the geodesic satisfying $\gamma(0) = x$ and $\gamma'(0) = \eta_x$, and $[0, 1] \subset I$ is an open interval.

\subsection{The existing Riemannian proximal Newton method} \label{sec04}

The existing Riemannian proximal Newton method (RPN) proposed in~\cite{SAHJV2024} is stated in Algorithm~\ref{alg:RPN}.
\begin{algorithm}[h]
\caption{A Riemannian proximal Newton method (RPN)}
\label{alg:RPN}
\begin{algorithmic}[1] 
\Require A $d$-dimensional embedded submanifold $\M$ of $\mathbb{R}^n$, an initial iterate $x_0 \in \mathcal{M}$, a parameter $t > 0$;
\For{$k = 0,1,\dots$}
\State \label{alg:RPN:st01} Compute $v(x_k)$ by solving 
\begin{equation} \label{eq:subforv}
v(x_k) = \argmin_{v \in \T_{x_k} \M}\ f(x_k) + \inner[]{\nabla f(x_k)}{v} + \frac{1}{2t} \|v\|_{\F}^2 + h(x_k + v).
\end{equation}
\State \label{alg:RPN:st02} Find $u(x_k) \in \T_{x_k}\M$ by solving
\begin{equation}\label{3-3}
    J(x_k) [u(x_k)] = - v(x_k),
\end{equation}
where 
\begin{equation}\label{eq:Jsol}
J(x_k) = -\left[\I_n - \Lambda_{x_k} + t\Lambda_{x_k} (\nabla^2 f(x_k) - \mathcal{L}_{x_k})\right],
\end{equation}
$\Lambda_{x_k} = M_{x_k} -  M_{x_k} B_{x_k} H_{x_k} B_{x_k}^{\T} M_{x_k}$, $H_{x_k} = \left(B_{x_k}^{\T} M_{x_k} B_{x_k}\right)^{-1}$, 
$B_{x_k}$ is an orthonormal basis of $\N_{x_k} \mathcal{M}$, 
$\mathcal{L}_{x_k}(\cdot) = \mathcal{W}_{x_k}\big(\cdot,B_{x_k} \lambda(x_k)\big)$, $\mathcal{W}_{x_k}$ denotes the Weingarten map, 
$\lambda(x_k)$ is the Lagrange multiplier~\eqref{3-1} at $x_k$, and $M_{x_k} \in \partial \prox_{t h}(x_k - t [\nabla f(x_k)  + B_{x_k} \lambda(x_k)])$ is a diagonal matrix defined by
\begin{equation} \label{eq:SpecialM}
(M_{x_k})_{i i} = 
\left\{
\begin{array}{cc}
    0 & \hbox{ if $|x_k - t [\nabla f(x_k)  + B_{x_k} \lambda(x_k)]|_i \leq t \mu$; } \\
    1 & \hbox{ otherwise.}
\end{array}
\right.
\end{equation}
\State $x_{k+1} = R_{x_k}\left( u(x_k)\right)$;
\EndFor
\end{algorithmic}
\end{algorithm}
Step~\ref{alg:RPN:st01} of Algorithm~\ref{alg:RPN} computes the Riemannian proximal gradient direction which is defined in~\cite{CMSZ2019,HuaWei2021}. It is proven that the direction $v(x_k)$ is a descent direction at $x_k$. 
A semismooth Newton algorithm has been proposed in~\cite{CMSZ2019} to solve~\eqref{eq:subforv} efficiently and it usually takes 2-3 iterations on average to find a high accurate solution~\cite{CMSZ2019,HuaWei2019,HuaWei2019b}. Specifically, noting that $v \in \T_{x_k} \mathcal{M}$ is equivalent to $B_{x_k}^T v = 0$, the KKT condition of~\eqref{eq:subforv} therefore is given by
\begin{equation} \label{eq:KKT01}
\partial_v \mathscr{L}_k(v, \lambda) = 0, \hbox{ and } B_{x_k}^T v = 0, 
\end{equation}
where $\mathscr{L}_k(v, \lambda) = f(x_k) + \inner[]{\nabla f(x_k)}{v} + \frac{1}{2t} \|v\|_{\F}^2 + h(x_k + v) + \lambda^T B_{x_k}^{T} v$ is the Lagrange function and $\lambda \in \mathbb{R}^{n-d}$ is the Lagrange multiplier. 
Equations in~\eqref{eq:KKT01} yield that
\begin{equation}\label{3-1}
v=\prox_{t h} \bigl( x_k - t \left[\nabla f(x_k) + B_{x_k} \lambda \right] \bigr) - x_k, \hbox{ and } B_{x_k}^{T} v = 0,
\end{equation}
where $ \mathrm{prox}_{t h}(z) $ denotes the proximal mapping of $t h$, i.e., 
\begin{equation} \label{eq:EProx}
\mathrm{prox}_{t h}(z) = \argmin_{x \in \mathbb{R}^n} \frac{1}{2} \|x - z\|_{\F}^2 + t h(x) = \max(|z| - t \mu, 0) \odot \sign(z).
\end{equation}
An equation of $\lambda$ given by
\begin{equation} \label{eq:Psilambda}
B_{x_k}^T \left( \prox_{t h} \bigl( x_k - t \left[\nabla f(x_k) + B_{x_k} \lambda \right] \bigr) - x_k \right) = 0
\end{equation}
follows from~\eqref{3-1}. By the semismoothness of $\prox_{t h}$, Equation~\eqref{eq:Psilambda} can be solved efficiently by a semismooth Newton method and $v(x_k)$ is computed by the first equation in~\eqref{3-1}.

Step~\ref{alg:RPN:st02} of Algorithm~\ref{alg:RPN} computes the Riemannian proximal Newton direction by solving a Newton equation in~\eqref{3-3} motivated by the semismooth Newton method~\cite{QS2006,XLWZ2018}. Specifically, the search direction $v(x_k)$ is semismooth with respect to $x_k$, as proven in~\cite{SAHJV2024}. The linear operator $J(x_k)$ is motivated by the generalized Jacobi of $v$ at $x_k$. In particular, when $x_k$ is at a stationary point $x_*$, $J(x_*)$ is a generalized Jacobi of $v$ at $x_*$. It is proven in~\cite{SAHJV2024} that Algorithm~\ref{alg:RPN} converges superlinearly locally under certain reasonable assumptions.

It has been proven in~\cite[Lemma~3.2]{SAHJV2024} that the $i$-th entry in $x_k + v(x_k)$ is zero if and only if $(M_{x_k})_{ii}$ is zero. We say that an index $i$ is in the support of $x_k + v(x_k)$ if the $i$-th entry of $x_k + v(x_k)$ is nonzero. Without loss of generality, we assume that the first $j_k$ entries of $x_k + v(x_k)$ are the support. 
It follows that $M_{x_k} = \begin{pmatrix} I_{j_k} & \\ & 0_{n-j_k,n-j_k} \end{pmatrix}$. Given any $A_{x_k} \in \mathbb{R}^{n \times m}$, define $\bar{\cdot}$ and $\hat{\cdot}$ operator as a partition of $A_{x_k}$, i.e., $\bar{A}_{x_k} \in \mathbb{R}^{j_k \times m}$ is the first $j_k$ rows of $A_{x_k}$ and $\hat{A}_{x_k} \in \mathbb{R}^{(n - j_k) \times m}$ is the last $n-j_k$ rows of $A_{x_k}$. Therefore, $A_{x_k} = \begin{pmatrix} \bar{A}_{x_k} \\ \hat{A}_{x_k} \end{pmatrix}$. It follows that
\begin{gather*}
x_k = \begin{pmatrix}
	\bar{x}_k \\
	\hat{x}_k
\end{pmatrix},
v(x_k) = \begin{pmatrix}
 \bar{v}(x_k) \\
 \hat{v}(x_k)	
 \end{pmatrix},
u(x_k) = \begin{pmatrix}
	\bar{u}(x_k) \\
	\hat{u}(x_k)
\end{pmatrix},
B_{x_k} = \begin{pmatrix}
	\bar{B}_{x_k} \\
	\hat{B}_{x_k}
\end{pmatrix}, 
\\
\nabla f(x_k) = \begin{pmatrix}
	\overline{\nabla f(x_k)} \\
	\widehat{\nabla f(x_k)}
\end{pmatrix},
\nabla^2 f(x_k) = \begin{pmatrix}
	H_{x_k}^{(11)} & H_{x_k}^{(12)} \\
	H_{x_k}^{(21)} & H_{x_k}^{(22)}
\end{pmatrix},
\mathcal{L}_{x_k} = \begin{pmatrix}
	L_{x_k}^{(11)} & L_{x_k}^{(12)} \\
	L_{x_k}^{(21)} & L_{x_k}^{(22)}
\end{pmatrix}, \\
\mathfrak{B}_{x_k} = \nabla^2 f(x_k) - \mathcal{L}_{x_k} =  \begin{pmatrix}
	\mathfrak{B}_{x_k}^{(11)} & \mathfrak{B}_{x_k}^{(12)} \\
	\mathfrak{B}_{x_k}^{(21)} & \mathfrak{B}_{x_k}^{(22)}
\end{pmatrix},
\mathcal{B}_{x_k} = \mathfrak{B}_{x_k}^{(11)} = H_{x_k}^{(11)} - L_{x_k}^{(11)},
\end{gather*}
where the partitions are based on the size of nonzero entries, i.e., $\bar{x}_k \in \mathbb{R}^{j_k}$, $\bar{v}(x_k) \in \mathbb{R}^{j_k}$, $\bar{u}(x_k) \in \mathbb{R}^{j_k}$, $\bar{B}_{x_k} \in \mathbb{R}^{j_k \times (n - d)}$, $H_{x_k}^{(11)} \in \mathbb{R}^{j_k \times j_k}$, and $L_{x_k}^{(11)} \in \mathbb{R}^{j_k \times j_k}$. It follows from~\eqref{eq:Jsol} that
\begin{equation*}
J(x_k) = - 
\begin{pmatrix}
	\bar{B}_{x_k} \bar{B}_{x_k}^{\dagger} + t (I_{j_k} - \bar{B}_{x_k} \bar{B}_{x_k}^{\dagger}) \mathcal{B}_{x_k} & t (I_{j_k} - \bar{B}_{x_k} \bar{B}_{x_k}^{\dagger}) \mathfrak{B}_{x_k}^{(12)} \\
	0_{(n - j_k) \times j_k} & I_{n - j_k}
\end{pmatrix},
\end{equation*}
where $\dagger$ denotes the pseudo-inverse and it is assumed that $\bar{B}_{x_k}$ has full column rank. Therefore, the direction $u(x_k)$ is obtained by solving
\begin{align} \label{eq:10}
\left\{
\begin{array}{c}
	[\bar{B}_{x_k} \bar{B}_{x_k}^{\dagger} + t (I_{j_k} - \bar{B}_{x_k} \bar{B}_{x_k}^{\dagger}) \mathcal{B}_{x_k}] \bar{u}(x_k) = \bar{v}(x_k) - t (I_{j_k} - \bar{B}_{x_k} \bar{B}_{x_k}^{\dagger}) \mathfrak{B}_{x_k}^{(12)} \hat{u}(x_k) \\
	\hat{u}(x_k) = \hat{v}(x_k)
\end{array}
\right.
.
\end{align}
Thus, only a linear system with the size of the nonzero entries in $x_k + v(x_k)$ needs to be solved for the Newton direction $u(x_k)$. This implies the sparser the solution is, the more efficient computing the superlinear direction $u(x_k)$ is. As shown in Section~\ref{sec01}, Equation~\eqref{eq:10} can be reformulated as a quadratic optimization problem under reasonable assumption even if the coefficient matrix is not symmetric.

An optimality condition of Problem~\eqref{eq:F}, stated in Proposition~\ref{p01}, is given in~\cite{SAHJV2024}. 
\begin{proposition} \label{p01}
If $x_* = \begin{pmatrix}
\bar{x}_* \\ 0
\end{pmatrix}$
is a local minimizer with $\bar{x}_* \in \mathbb{R}^j$ and $\bar{B}_{x_*}$ has full column rank. Then $v(x_*) = 0$ and $\mathcal{B}_{x_*} \succeq 0$ on the subspace $\mathfrak{L}_{x_*}$, where $\mathfrak{L}_x$ is defined by
$\mathfrak{L}_x = \{w: \bar{B}_{x}^{T} w = 0\}$.
\end{proposition}
To prove the local superlinear convergence of the proposed method, we assume that the linear operator $\mathcal{B}_{x_*}$ is positive definite on $\mathfrak{L}_{x_*}$, which is slightly stronger than that in Proposition~\ref{p01}, as shown in Section~\ref{sec03}.

\section{A Riemannian Proximal Newton-CG Method} \label{sec09}

The proposed Riemannian proximal Newton-CG method (RPN-CG) is stated in Algorithm~\ref{alg:IRPN}. But we invite the readers to first read the more reader-friendly descriptions in the next few paragraphs and to refer to the pseudocode in Algorithm~\ref{alg:IRPN} when needed.

\begin{algorithm}
\caption{A Riemannian proximal Newton-CG method (RPN-CG)}
\label{alg:IRPN}
\begin{algorithmic}[1] 
\Require A $d$-dimensional embedded submanifold $\M$ of $\mathbb{R}^n$; $x_0 \in \mathcal{M}$; an initial step size $\alpha_{\mathrm{init}}$; line search parameters $\rho_1, \rho_2 \in (0, 1)$;  two parameters $t_{\max} \geq t_{\min} > 0$ for proximal mapping; $\varpi_1 > 1$ and $\varpi_2 \in (0, 1)$ for the update of $t$; $\vartheta > 0$, $\gamma > 0$, $\tau > 0$, $\theta > 0$, and $\kappa \in (0, 1)$ for Algorithm~\ref{alg:tCG};
\State Set $\mathrm{flag} = 0$ and $k = 0$;
\Loop
\State \label{a:RNG:st01} Compute 
$v(x_k) = \argmin_{v \in \T_{x_k} \M}\ f(x_k) + \inner[]{\nabla f(x_k)}{v} + \frac{1}{2t_k} \|v\|_{\F}^2 + h(x_k + v)$;
\State \label{a:RNG:st02} 
Invoke Algorithm~\ref{alg:tCG} with $[\ell_{x_k,t_k}, \mathcal{B}_{x_k}, \mathfrak{B}_{x_k}, t_k, \vartheta, \gamma, \tau, \theta, \kappa, v(x_k), P_{x_k}, G_{x_k}]$ to approximately solve
\begin{equation} \label{eq:09}
\argmin_{ \bar{B}_{x_k}^T w = 0 } \inner[]{\ell_{x_k,t_k}}{w} + \frac{1}{2} \inner[]{w}{\mathcal{B}_{x_k} w}
\end{equation}
and the output is denoted by $(w(x_k), \mathrm{status})$,
where $\ell_{x,t} = \frac{1}{t} ( -I_{j} + t \mathcal{B}_{x} ) \bar{v}(x) + \mathfrak{B}_{x}^{(12)} \hat{v}(x)$, $P_x = I_n - \bar{B}_x \bar{B}_x^{\dagger}$ is the orthogonal projection onto the null space of $\bar{B}_x$, and $G_{x}(u) = f(x) + \inner[]{\nabla f(x)}{u} + \frac{1}{2} \inner[]{u}{\mathfrak{B}_x u} + \frac{\tau}{2} \|\hat{u}(x)\|_{\F}^2 + h(x + u)$;
\State \label{a:RNG:st03} Compute $d(x_k)$ by
$d(x_k) = \begin{pmatrix}
	\bar{d}(x_k) \\
	\hat{d}(x_k)
\end{pmatrix}
 = \begin{pmatrix}
	\bar{v}(x_k) + w(x_k) \\
	\hat{v}(x_k)
\end{pmatrix}$;
\If { $(4 + 1/t_k) \|d(x_k)\|_{\F} < \|v(x_k)\|_{\F}$ or $\mathrm{status} = 'early1'$ } \label{a:RNG:st17}
\State $t_{k+1} = \max( \varpi_2 t_k, t_{\min})$; \label{a:RNG:st18}
\ElsIf {$\mathrm{status} \neq 'sup'$} \label{a:RNG:st19}
\State $t_{k+1} = \min(\varpi_1 t_k, t_{\max})$; \label{a:RNG:st20}
\EndIf \label{a:RNG:st21}
\If {$\mathrm{status} \neq 'sup'$ and $\mathrm{flag} \neq 1$}
\State \label{a:RNG:st04} Set $\mathrm{flag} = 0$ and $\alpha_k = \alpha_{\mathrm{init}}$;
\While { $F(R_{x_k}( \alpha_k d(x_k) )) > F(x_k) - \alpha_k \rho_1 \|d(x_k)\|_{\F}^2$ } \label{a:RNG:st15}
\State $\alpha_k = \alpha_k \rho_2$;
\EndWhile \label{a:RNG:st05}
\State Set $x_{k+1} = R_{x_k}\left( \alpha_k d(x_k)\right)$ and $k = k + 1$; \label{a:RNG:st06}
\Else
\State Set $\mathrm{flag} = \mathrm{flag} + 1$; \label{a:RNG:st07}
\If {$\mathrm{flag} = 1$}
\State Set $x_{k+1} = R_{x_k}\left( d(x_k)\right)$ and $k = k + 1$; \Comment{The step size is one} \label{a:RNG:st09}
\Else
\State Set $x_{k+1} = R_{x_k}\left( d(x_k)\right)$ and $\mathrm{flag} = 0$; \label{a:RNG:st10}
\If { $F(x_{k+1}) > F(x_{k-1}) - \rho_1 \|v(x_{k-1})\|_{\F}^2$ } \Comment{Not sufficiently descent in two steps} \label{a:RNG:st11}
\State Set $\alpha_{k-1} = \alpha_{\mathrm{init}}$; \Comment{Backtracking at $x_{k-1}$} \label{a:RNG:st12}
\While { $F(R_{x_{k-1}}( \alpha_{k - 1} d(x_{k-1}))) > F(x_{k-1}) - \alpha_{k - 1} \rho_1 \|d(x_{k-1})\|_{\F}^2$ } \label{a:RNG:st16}
\State $\alpha_{k-1} = \alpha_{k-1} \rho_2$;
\EndWhile
\State Set $x_{k} = R_{x_{k-1}}\left( \alpha_{k-1} d(x_{k-1})\right)$; \label{a:RNG:st13}
\Else \Comment{Sufficiently descent in two steps}
\State Set $k = k + 1$; \label{a:RNG:st14}
\EndIf
\EndIf \label{a:RNG:st08}
\EndIf
\EndLoop
\end{algorithmic}
\end{algorithm}

Step~\ref{a:RNG:st01} of the RPN-CG in Algorithm~\ref{alg:IRPN} is the same as Step~\ref{alg:RPN:st01} of the RPN in Algorithm~\ref{alg:RPN}, which computes a proximal gradient direction. 
RPN-CG does not compute the proximal Newton direction $u(x_k)$ exactly. Instead, an approximation $d(x_k)$ of $u(x_k)$, computed from Step~\ref{a:RNG:st02} to Step~\ref{a:RNG:st03}, is used as the search direction of the RPN-CG in Algorithm~\ref{alg:IRPN}. It is shown later in Lemma~\ref{le03} that $d(x_k)$ is a descent direction. It follows that a line search strategy from Step~\ref{a:RNG:st04} to Step~\ref{a:RNG:st05} terminates in finitely many iterations and the next iterate $x_{k+1}$ is computed by Step~\ref{a:RNG:st06}. Moreover, if the iterate $x_k$ is sufficiently close to a local minimizer $x_*$ of $F$ in~\eqref{eq:F}, then the search direction $d(x_k)$ is a good approximation of the proximal Newton direction $u(x_k)$ in~\eqref{3-3}, as shown in Section~\ref{sec03}. If the step size one is used, then the convergence rate is superlinear. However, due to the nonsmoothness of the objective function and the curvature of the retraction, the search direction $d(x_k)$ with step size one may not give a decrease in the function $F$. But it will be shown later in Proposition~\ref{p02} that the function value is sufficiently descent with step size one every two steps when $x$ is sufficiently close to~$x_*$, i.e., 
\begin{equation} \label{eq:67}
F(x_{k+1}) \leq F(x_{k-1}) - \rho_1 \|v(x_{k-1})\|_{\F}^2.
\end{equation}
This motivates us to design Steps~\ref{a:RNG:st07} to~\ref{a:RNG:st08} of Algorithm~\ref{alg:IRPN}. If $\mathrm{flag} = 0$, then the backtracking algorithm from Step~\ref{a:RNG:st04} to~\ref{a:RNG:st06} is used. 
This happens usually when $x$ is not close to $x_*$. If $\mathrm{flag} = 1$, then the step size one is used without checking the sufficient descent condition, see Step~\ref{a:RNG:st09}. If $\mathrm{flag} = 2$, then the step size one is again used without checking the sufficient descent condition, see Step~\ref{a:RNG:st10}. In this case, the previous two iterates are computed using step size one. Therefore, we check if~\eqref{eq:67} holds. If it does, then the previous two iterates are valid, see Step~\ref{a:RNG:st14}. Otherwise, a backtracking algorithm is performed at $x_{k-1}$, see Steps~\ref{a:RNG:st12} to~\ref{a:RNG:st13}.

Some theoretical analyses in global and local superlinear convergence rely on a sufficiently small~$t$. However, using a small value of $t$ prevents fast convergence of the algorithm. Therefore, the parameter~$t$ is updated dynamically by Steps~\ref{a:RNG:st17} to~\ref{a:RNG:st21} without harming the theoretical results. Note that $\{t_k\}$ satisfies $t_{\min} \leq t_k \leq t_{\max}$ for all $k$.

The truncated conjugate gradient method (tCG) in Algorithm~\ref{alg:tCG} is used to approximately solve Problem~\eqref{eq:09} in Step~\ref{a:RNG:st02} of Algorithm~\ref{alg:IRPN}, or equivalently, is viewed as a method for approximately solving $(P_{x_k} \circ \mathcal{B}_{x_k} \circ P_{x_k}) w = P_{x_k} (\ell_{x_k,t_k})$.
Note that $\mathcal{B}_{x_k}$ is symmetric by the property of Hessian and Weingarten map and is further positive definite over the subspace $\mathfrak{L}_{x_*} = \{w:\bar{B}_{x_*}^T w = 0\}$ under certain assumptions as shown in~\cite[Proposition~3.13]{SAHJV2024}.

The early termination condition in Step~\ref{a:tCG:st04} of Algorithm~\ref{alg:tCG} is standard and has been used in~\cite{DS83NewtontCG}. 
The early stopping conditions in Steps~\ref{a:tCG:st06},~\ref{a:tCG:st02} and~\ref{a:tCG:st03} of Algorithm~\ref{alg:tCG} are new. They are used to guarantee the descent of the search direction $d(x_k)$. Moreover, it is shown in Section~\ref{sec03} that if $x_k$ is sufficiently close to a minimizer $x_*$, then these early termination criteria would not take effect and tCG stops due to the accurate condition in Step~\ref{a:tCG:st01}. This is crucial for the local superlinear convergence.

\begin{remark} \label{re01}
When the manifold $\mathcal{M}$ is the Euclidean space $\mathbb{R}^n$, Algorithm~\ref{alg:IRPN} does not become any existing Euclidean algorithm as far as we know. Specifically, in this case, the $v(x_k)$ is the Euclidean proximal gradient direction. It follows that $\overline{\nabla f(x_k)} = - \mu \mathrm{sign}( \bar{x}_k + \bar{v}(x_k)) - \frac{1}{t} \bar{v}(x_k)$. Using this equation for~\eqref{eq:09} yields that $w(x_k)$ can be viewed as an approximation of
\begin{align}
\argmin_{w \in \mathbb{R}^{j_k}} & \inner[]{\nabla f(x_k)}{ \begin{pmatrix}
    \bar{v}(x_k) + w \\
    \hat{v}(x_k)
\end{pmatrix} }
+ \frac{1}{2} \inner[]{\begin{pmatrix}
    \bar{v}(x_k) + w \\
    \hat{v}(x_k)
\end{pmatrix}}{
\nabla^2 f(x_k)
\begin{pmatrix}
    \bar{v}(x_k) + w \\
    \hat{v}(x_k)
\end{pmatrix}} \nonumber \\
& \qquad\qquad \qquad + \mu \mathrm{sign}
\left( x_k + \begin{pmatrix}
    \bar{v}(x_k) \\
    \hat{v}(x_k)
\end{pmatrix} \right)^T \left( x_k + \begin{pmatrix}
    \bar{v}(x_k) + w \\
    \hat{v}(x_k)
\end{pmatrix}
\right) \label{eq:71}
\end{align}
In other words, one can view $\bar{v}(x_k) + w$ as an approximation of the Newton direction at iterate $x_k$ with underlying sparsity structure fixed. 
The closest approach given in~\cite[Algorithm~1]{BIM2023} first uses a Euclidean proximal mapping to obtain an iterate $y_k:=x_k + v(x_k)$. The underlying manifold structure, such as sparsity, is identified by $y_k$. A Newton-CG step with the manifold structure fixed is then used at $y_k$, not $x_k$, which is different from Algorithm~\ref{alg:IRPN}.
\end{remark}

\begin{algorithm}
\caption{A truncated conjugate gradient algorithm (tCG)}
\label{alg:tCG}
\begin{algorithmic}[1] 
\Require $[\ell_{x,t}, \mathcal{B}, \mathfrak{B}, t, \vartheta, \gamma, \tau, \theta, \kappa, v(x), P_x, G_x]$, where $g \in \mathbb{R}^j$, $\mathcal{B} \in \mathbb{R}^{j \times j}$, $\mathfrak{B} \in \mathbb{R}^{n \times n}$, $\vartheta > 0$, $\gamma > 0$, $\tau > 0$, $\theta > 0$, and $\kappa \in (0, 1)$; Note that $\mathfrak{B} = \begin{pmatrix} \mathcal{B} & \mathfrak{B}^{(12)} \\ \mathfrak{B}^{(21)} & \mathfrak{B}^{(22)} \end{pmatrix}$;
\Ensure $(w(x), \mathrm{status})$; 
\If { $G_x(v(x)) > G_x(0)$ } \label{a:tCG:st06}
\State return $w(x) = 0$ and $\mathrm{status} = 'early1'$;
\EndIf
\State $z = \mathfrak{B} v(x)$; 
\If { $\inner[]{v(x)}{z} + \tau \|\hat{v}(x)\|_{\F}^2 < \gamma \|v(x)\|_{\F}^2$ } \label{a:tCG:st02}
\State return $w(x) = 0$ and $\mathrm{status} = 'early2'$;
\EndIf \label{a:tCG:st07}
\State $w_0 = 0$, $r_0 = P_x (\ell_{x,t})$, $o_0 = -r_0$, $\delta_0 = \inner[]{r_0}{r_0}$, $t_{\min} = z$; 
\For {$i = 0, 1, \ldots$}
\State $p_i = \mathcal{B} o_i$ and $q_i = P_x (p_i)$;
\If { $\inner[]{o_i}{q_i} \leq \vartheta \delta_i$} \label{a:tCG:st04}
\State return $w(x) = w_i$ and $\mathrm{status} = 'neg'$;\label{a:tCG:st05}
\EndIf
\State $\alpha_i = \frac{\inner[]{r_i}{r_i}}{ \inner[]{o_i}{ q_i} }$, $w_{i+1} = w_i + \alpha_i o_i$, $r_{i+1} = r_i + \alpha_i q_i$;
\State $d_{i+1} = \begin{pmatrix} \bar{v}(x) + w_{i+1} \\ \hat{v}(x) \end{pmatrix}$, $t_{i+1} = t_i + \alpha_i \begin{pmatrix} p_i \\ \mathfrak{B}_{21} o_i \end{pmatrix}$; \Comment{Note that $t_{i+1} = \mathfrak{B} d_{i+1}$} 
\If { $ \inner[]{d_{i+1}}{ t_{i+1}}  + \tau \|\hat{v}(x)\|_{\F}^2 < \gamma \|d_{i+1}\|_{\F}^2 $ or 
$G_{x}\left( d_{i+1} \right) > G_{x}(0)$ }  \label{a:tCG:st03} 
\State return $w(x) = w_i$ and $\mathrm{status} = 'early3'$;
\EndIf \label{a:tCG:st08}
\State $\beta_{i+1} = \frac{ \inner[]{r_{i+1}}{r_{i+1}} }{ \inner[]{r_i}{r_i} }$, $o_{i+1} = -r_{i+1} + \beta_{i+1} o_i $;
\State $\delta_{i+1} = \inner[]{r_{i+1}}{r_{i+1}} + \beta_{i+1}^2 \delta_i$; \Comment{ $\delta_{i+1} = \inner[]{o_{i+1}}{o_{i+1}}$ }
\State $i = i+1$;
\If { $\|r_{i}\|_{\F} \leq \|r_0\|_{\F} \min( \|r_0\|_{\F}^{\theta}, \kappa )$ }  \label{a:tCG:st01}
\State return $w(x) = w_{i}$, and $\mathrm{status} = 'lin'$ if $\|r_0\|_{\F}^{\theta} > \kappa$ and $\mathrm{status} = 'sup'$ otherwise;
\EndIf
\EndFor
\end{algorithmic}
\end{algorithm}

\subsection{Global convergence analysis} \label{sec02}

The global convergence is established under Assumptions~\ref{as01}. Assumption~\ref{as01} is standard and has been used in~\cite{CMSZ2019,HuaWei2019, HuaWei2019b,SAHJV2024}.
\begin{assumption} \label{as01}
The function $f$ is twice continuously differentiable and its gradient $\nabla f$ is Lipschitz continuous with Lipschitz constant $L_f$.
\end{assumption}

Lemma~\ref{le03} shows that $d(x_k)$ with the tCG in Algorithm~\ref{alg:tCG} is a descent direction. This result is analogous to the existing result in~\cite[Lemma~5.2]{CMSZ2019}. However, the proof is different in the sense that it relies on the termination conditions in Steps~\ref{a:tCG:st06},~\ref{a:tCG:st02} and~\ref{a:tCG:st03} of Algorithm~\ref{alg:tCG}. Specifically, the search direction $d(x_k) = \begin{pmatrix} \bar{v}(x_k) + w(x_k) \\ \hat{v}(x_k) \end{pmatrix}$ with the output $w(x_k)$ of Algorithm~\ref{alg:tCG} satisfies
either $d(x_k) = v(x_k)$ by Steps~\ref{a:tCG:st06},~\ref{a:tCG:st02} or
\begin{align}
\inner[]{d(x_k)}{\mathfrak{B}_{x_k} d(x_k)} + \tau \| \hat{v}(x_k)\|_{\F}^2 \geq& \gamma \|d(x_k) \|_{\F}^2 
\hbox{ and } \label{eq:17} \\
G_{x_k}(d(x_k)) \leq& G_{x_k}(0) \label{eq:18}
\end{align}
by Step~\ref{a:tCG:st03}. If $d(x_k) = v(x_k)$ holds, then the theoretical results in~\cite{CMSZ2019} can be used and the descent property of $d(x_k)$ follows. Otherwise, inequality~\eqref{eq:17} ensures that $\mathfrak{B}_{x_k} + \tau \mathrm{diag}(0_{j \times j}, I_{n-j})$ shows positive definiteness along direction $d(x_k)$. Inequality~\eqref{eq:18} plays the same role as~\cite[the first inequality on Page 224]{CMSZ2019}. These two facts are sufficient to show the descent property of $d(x_k)$, as shown in Lemma~\ref{le03}. Therefore, the line search in Steps~\ref{a:RNG:st15} and~\ref{a:RNG:st16} terminates in finite steps. 

\begin{lemma} \label{le03}
Suppose Assumption~\ref{as01} holds. 
Then there exist two constants $\bar{\alpha} > 0$ and $\rho_1 > 0$, both independent of $k$, such that the sufficient descent condition
\begin{equation} \label{eq:25}
F(R_{x_k}( \alpha d(x_k) )) \leq F(x_k) - \alpha \rho_1 \|d(x_k)\|_{\F}^2
\end{equation}
holds for all $\alpha \in (0, \bar{\alpha})$. Therefore, the accepted step size $\alpha_k$ satisfies $\alpha_k \geq \bar{\alpha} \rho_2$ for all~$k$.
\end{lemma}
\begin{proof}
If the condition in Step~\ref{a:tCG:st06} or Step~\ref{a:tCG:st02} holds, then $d(x_k) = v(x_k)$. It follows from~\cite[Lemma~5.2]{CMSZ2019} that there exists $\bar{\alpha} > 0$ such that for any $\alpha \in (0, \bar{\alpha})$, it holds that
\begin{equation} \label{eq:15}
F(R_{x_k}( \alpha d(x_k) )) \leq F(x_k) - \frac{\alpha}{2 t_k} \|d(x_k)\|_{\F}^2 \leq F(x_k) - \frac{\alpha}{2 t_{\max}} \|d(x_k)\|_{\F}^2.
\end{equation}
Otherwise, Algorithm~\ref{alg:tCG} terminates when the conditions in Step~\ref{a:tCG:st03} fail, i.e., 
\eqref{eq:17} and~\eqref{eq:18} hold.
From~\eqref{eq:18}, we have
\begin{equation} \label{eq:19}
h(x_k) \geq \inner[]{ \nabla f(x_k) }{ d(x_k)} + \frac{1}{2} \inner[]{d(x_k)}{ \mathfrak{B}_{x_k} d(x_k)} + \frac{\tau}{2} \|\hat{d}(x_k)\|_{\F}^2 + h(x_k + d(x_k)).
\end{equation}
By the convexity of $h$, it holds that for any $\alpha \in [0, 1]$,
\begin{equation} \label{eq:20}
h(x_k + \alpha d(x_k)) - h(x_k) \leq \alpha ( h(x_k + d(x_k)) - h(x_k) ).	
\end{equation}
Therefore, we have that for $\alpha \in [0, 1]$
\begin{align}
&\alpha \inner[]{\nabla f(x_k)}{ d(x_k)} + h(x_k + \alpha d(x_k)) - h(x_k) \nonumber \\
\leq& \alpha \left( \inner[]{\nabla f(x_k)}{ d(x_k)} + h(x_k + d(x_k)) - h(x_k) \right) \nonumber \\
\leq& - \frac{\alpha}{2} \left( \inner[]{d(x_k)}{ \mathfrak{B}_{x_k} d(x_k)} + \tau \|\hat{d}(x_k)\|_{\F}^2 \right)
\leq - \frac{\alpha}{2} \gamma \|d(x_k)\|_{\F}^2, \label{eq:21}
\end{align}
where the first, second, and last inequalities respectively follow from~\eqref{eq:20},~\eqref{eq:19}, and~\eqref{eq:17}.
By the Lipschitz continuity of $\nabla f$, we have
\begin{align}
&f(R_{x_k}(\alpha d(x_k))) - f(x_k) \nonumber \\
\leq& \inner[]{\nabla f(x_k)}{  R_{x_k}(\alpha d(x_k)) - x_k} + \frac{L_f}{2} \| R_{x_k}(\alpha d(x_k)) - x_k \|_{\F}^2 \nonumber\\
=& \inner[]{\nabla f(x_k)}{ ( R_{x_k}(\alpha d(x_k)) - x_k^+ + x_k^+ - x_k)} + \frac{L_f}{2} \| R_{x_k}(\alpha d(x_k)) - x_k \|_{\F}^2 \nonumber \\
\leq& C_{R_2} \|\nabla f(x_k)\|_{\F} \|\alpha d(x_k)\|_{\F}^2 + \alpha \inner[]{\nabla f(x_k)}{ d(x_k)} + \frac{L_f C_{R_1}^2}{2} \|\alpha d(x_k)\|_{\F}^2, \nonumber \\
\leq& \alpha \inner[]{\nabla f(x_k)}{ d(x_k)} + (C_{R_2} U_g + L_f C_{R_1}^2/2) \alpha^2 \|d(x_k)\|_{\F}^2 \label{eq:28}
\end{align}
where $x_k^+ = x_k + \alpha d(x_k)$, the second inequality follows from~\eqref{eq:26} and~\eqref{eq:27}, and $U_g = \max_{x \in \mathcal{M}} \|\nabla f(x)\|_{\F}$. It follows that
\begin{align*}
&F(R_{x_k}(\alpha d(x_k))) - F(x_k) \leq \alpha \inner[]{\nabla f(x_k)}{ d(x_k)} + (C_{R_2} U_g + L_f C_{R_1}^2/2) \alpha^2 \|d(x_k)\|_{\F}^2 \\
&\qquad \qquad + h(R_{x_k}( \alpha d(x_k) )) - h(x_k^+) + h(x_k^+) - h(x_k) \\
\leq& \alpha \inner[]{\nabla f(x_k)}{ d(x_k)} + (C_{R_2} U_g + L_f C_{R_1}^2/2) \alpha^2 \|d(x_k)\|_{\F}^2 \\
&\qquad \qquad + L_h \| R_{x_k}( \alpha d(x_k) ) - x_k^+ \|_{\F} + h(x_k^+) - h(x_k) \\
\leq& \alpha \inner[]{\nabla f(x_k)}{ d(x_k)} + (C_{R_2} U_g + L_f C_{R_1}^2/2 + L_h C_{R_2}) \alpha^2 \|d(x_k)\|_{\F}^2 + h(x_k^+) - h(x_k) \\
\leq& \left( C_{R_2} U_g + \frac{L_f C_{R_1}^2}{2} + L_h C_{R_2} - \frac{\gamma}{2 \alpha} \right) \alpha^2 \|d(x_k)\|_{\F}^2,
\end{align*}
where the first inequality follows from~\eqref{eq:28}, the second inequality follows from the Lipschitz continuity of $h$ with Lipschitz constant $L_h$, the third inequality follows from~\eqref{eq:27}, and the fourth equation follows from~\eqref{eq:21}. Let $\bar{\alpha} = \gamma / (4 (C_{R_2} U_g + L_f C_{R_1}^2/2 + L_h C_{R_2}) )$. We have that for any $\alpha \in (0, \bar{\alpha})$, it holds that
\begin{equation} \label{eq:30}
F(R_{x_k}( \alpha d(x_k) )) \leq F(x_k) - \alpha \frac{\gamma}{4} \|d(x_k)\|_{\F}^2.
\end{equation}
The final result~\eqref{eq:25} follows from~\eqref{eq:15} and~\eqref{eq:30} with $\rho_1 = \min\left(\frac{1}{2 t_{\max}}, \frac{\gamma}{4}\right)$. It follows from the line search in Steps~\ref{a:RNG:st15} to~\ref{a:RNG:st05} that $\alpha_k \geq \bar{\alpha} \rho_2$ for all $k$.
\end{proof}

Lemma~\ref{le11} is used for the global convergence analysis in Theorem~\ref{th03} and implies that if $t_k$ is sufficiently small, then $F(R_{x_k}( \alpha d(x_k) )) \leq F(x_k) - \frac{1}{(4 + 1/t_k)^2} \alpha \rho_1 \|v(x_k)\|_{\F}^2$ holds.
\begin{lemma} \label{le11}
Suppose Assumption~\ref{as01} holds. If $t$ is sufficiently small, then it holds that $(4 + 1/t) \|d(x)\|_{\F} \geq \|v(x)\|_{\F}$ for any $x \in \mathcal{M}$.
\end{lemma}
\begin{proof}
Algorithm~\ref{alg:tCG} is the algorithm in~\cite[Minor Iteration]{DS83NewtontCG} with additional termination conditions. Therefore, the theoretical results about the truncated conjugate gradient algorithm in~\cite{DS83NewtontCG} still hold. By~\cite[Lemma~A.2]{DS83NewtontCG}, we have
\begin{equation} \label{eq:40}
	\inner[]{w(x)}{\ell_{x,t}} \leq - b_0 \|P_x (\ell_{x,t})\|_{\F}^2,
\end{equation}
where $b_0 = \min(1, 1 / \| \mathcal{B} \|_2)$. It follows from~\eqref{eq:40} that
\begin{align} \label{eq:69}
\inner[]{\bar{d}(x)}{-\ell_{x,t}} = \inner[]{\bar{v}(x) + w(x)}{-\ell_{x,t}} \geq \inner[]{\bar{v}(x)}{-\ell_{x,t}} + b_0 \|P_x (\ell_{x,t})\|_{\F}^2 \geq \inner[]{\bar{v}(x)}{-\ell_{x,t}}.
\end{align}
The left term of~\eqref{eq:69} satisfies
\begin{align*}
\inner[]{\bar{d}(x)}{-\ell_{x,t}} =& \inner[]{\bar{d}(x)}{ \frac{1}{t} ( I_{j} - t \mathcal{B}_{x} ) \bar{v}(x) - \mathfrak{B}_{x}^{(12)} \hat{v}(x) } \\
\leq& \|\bar{d}(x)\|_{\F} \left( \frac{1 + t b_1}{t}  \| \bar{v}(x) \|_{\F} + b_2 \|\hat{v}(x)\|_{\F}\right)
\leq (b_3 + \frac{1}{t}) \|\bar{d}(x)\|_{\F} \|v(x)\|_{\F},
\end{align*}
where $b_1 \geq \max_{x \in \mathcal{M}} \sigma_{\max} (\mathcal{B}_x)$, $b_2 \geq \max_{x \in \mathcal{M}} \sigma_{\max} (\mathfrak{B}_x^{(12)})$, and $b_3 = b_1 + b_2$. The right term of~\eqref{eq:69} satisfies
\begin{align*}
&\inner[]{\bar{v}(x)}{-\ell_{x,t}} = \inner[]{ \bar{v}(x) }{ \frac{1}{t} ( I_{j} - t \mathcal{B}_{x} ) \bar{v}(x) - \mathfrak{B}_{x}^{(12)} \hat{v}(x) } \\
\geq& b_1 \|\bar{v}(x)\|_{\F}^2 - b_2 \|\bar{v}(x)\|_{\F} \|\hat{v}(x)\|_{\F} = b_1 \|v(x)\|_{\F}^2 - (b_1 \|\hat{v}(x)\|_{\F} + b_2 \|\bar{v}(x)\|_{\F}) \|\hat{v}(x)\|_{\F} \\
\geq& b_1 \|v(x)\|_{\F}^2 - (b_1 + b_2) \|v(x)\|_{\F} \|\hat{v}(x)\|_{\F},
\end{align*}
when $t < 1 / (2 b_1)$. It follows that
$(b_3 + 1/t) \|\bar{d}(x)\|_{\F} \|v(x)\|_{\F} \geq b_1 \|v(x)\|_{\F}^2 - (b_1 + b_2) \|v(x)\|_{\F} \|\hat{v}(x)\|_{\F}$, which yields $(b_1 + b_2) \|\hat{v}(x)\|_{\F} + (b_3 + 1/t) \| \bar{d}(x) \|_{\F} \geq b_1 \|v(x)\|_{\F}$.
Therefore, we have $(b_4 + 1/t) \|d(x)\|_{\F} \geq b_1 \|v(x)\|_{\F}$, where $b_4 = b_1 + b_2 + b_3$. Choosing $b_1 = b_2 = \max\{1,\max_{x \in \mathcal{M}} \sigma_{\max} (\mathfrak{B}_x)\}$ yields $b_3 = 2 b_1$ and $b_4 = 4 b_1$. It follows that $(4 + 1/t) \|d(x)\|_{\F} \geq \|v(x)\|_{\F}$, which completes the proof.
\end{proof}

The global convergence is therefore established from Lemma~\ref{le03} and stated in Theorem~\ref{th03}.
\begin{theorem} \label{th03}
Suppose Assumption~\ref{as01} holds and $t_{\min}$ is sufficiently small such that the result of Lemma~\ref{le11} holds. Then it holds that 
\[
\liminf_{k \rightarrow \infty} \|v(x_k)\|_{\F} = 0.
\]
\end{theorem}
\begin{proof}
Since $F$ is continuous and $\mathcal{M}$ is compact, $F$ is bounded from below. Therefore, we have
\begin{align} \label{eq:70}
& \infty > \sum_{k = 0}^\infty (F(x_k) - F(x_{k+1})) \geq \sum_{k = 0}^{\infty} \min(\rho_1 \|v(x_k)\|_{\F}^2, \bar{\alpha} \rho_2 \rho_1 \|d(x_k)\|_{\F}^2 ),
\end{align}
where the last inequality follows from Lemma~\ref{le03}. Next, we claim that there exists an infinite subsequence $\{k_j\}$ such that $(4 + 1/t_k)\|d(x_k)\|_{\F} \geq \|v(x_k)\|_{\F}$. Otherwise, there exists $K>0$ such that $(4 + 1/t_k)\|d(x_k)\|_{\F} < \|v(x_k)\|_{\F}$ holds for all $k \geq K$. By Step~\ref{a:RNG:st17} of Algorithm~\ref{alg:IRPN}, $t_k$ is eventually equal to $t_{\min}$. Since $t_{\min}$ is sufficiently small such that $(4 + 1/t_{\min})\|d(x)\|_{\F} \geq \|v(x)\|_{\F}$ holds for any $x \in \mathcal{M}$, we have $(4 + 1 / t_k)\|d(x_k)\|_{\F} \geq \|v(x_k)\|_{\F}$ for $k$ sufficiently large, which conflicts with $(4 + 1/t_k)\|d(x_k)\|_{\F} < \|v(x_k)\|_{\F}, \forall k \geq K$.
It follows from~\eqref{eq:70} that $\infty > \sum_{j=1}^{\infty} \min(\rho_1, \bar{\alpha} \rho_2 \rho_1/(4 + 1/t_{\min})^2) \|v(x_{k_j})\|_{\F}^2$, which implies
$\liminf_{k \rightarrow \infty} \|v(x_k)\|_{\F} = 0$.
\end{proof}

\subsection{Local convergence analysis} \label{sec03}
 
This section shows that if an iterate $x_k$ is sufficiently close to a local minimizer $x_*$, then the sequence $\{x_k\}$ converges to $x_*$ and the rate is superlinear. 
The analysis is established under Assumptions~\ref{as02} and~\ref{as03}, which have been made in~\cite{SAHJV2024}. Assumption~\ref{as02} assumes that the support of $x + v$ does not change in a neighborhood of $x_*$
and has proven to be held when $\mathcal{M}$ is a Euclidean space under reasonable assumptions in~\cite{BIM2023,LFP2017}.
Assumption~\ref{as03} gives a condition that $x_*$ satisfies and it holds if the manifold is the unite sphere $\mathbb{S}^{n - 1} = \{x \in \mathbb{R}^n \mid x^T x = 1\}$ or the oblique manifold $\left(\mathbb{S}^{n - 1}\right)^{m}$. 

\begin{assumption}\label{as02}
There exists a neighborhood $\mathcal{U}_{x_*}$ of $x_* = [\bar{x}_*^T, 0^T]^T$ on $\mathcal{M}$ such that for any $x = [\bar{x}^T, \hat{x}^T]^T \in \mathcal{U}_{x_*}$, it holds that $\bar{x} + \bar{v} \neq 0$ and $\hat{x} + \hat{v} = 0$, where $v = [
\bar{v}^T, \hat{v}^T ]^T$ denotes the solution of the subproblem in Step~\ref{a:RNG:st01} of Algorithm~\ref{alg:IRPN} at $x$.
\end{assumption}

\begin{assumption}\label{as03}
 Let $B_{x_*}^{\T} = [\bar{B}_{x_*}^{\T}, \hat{B}_{x_*}^{\T}] $, where $\bar{B}_{x_*} \in \mathbb{R}^{j \times d}$ and $\hat{B}_{x_*} \in \mathbb{R}^{(n-j)\times d}$. It is assumed that $j \geq d$ and $\bar{B}_{x_*}$ is full column rank.
\end{assumption}

It is shown in~\cite[Proposition 3.13]{SAHJV2024} that under Assumption~\ref{as03}, the linear operator $H_{x_*}^{(11)} - L_{x_*}^{(11)}$ is positive semidefinite on $\{w \mid \bar{B}_{x_*}^T w = 0 \}$ if $x_*$ is a local minimizer of $F$. Assumption~\ref{as04} is slightly stronger than the necessary optimality condition. This assumption is analogous to that used in the superlinear convergence analysis of the Newton method for smooth problems. See details in~\cite[Theorem~2.3 and Theorem~3.5]{NocWri2006}. 
\begin{assumption} \label{as04}
The linear operator $\mathcal{B}_{x_*}$ is positive definite on the subspace $\mathfrak{L}_{x_*} = \{w \mid \bar{B}_{x_*}^T w = 0 \}$.
\end{assumption}

Assumptions~\ref{as02}, \ref{as03} and~\ref{as04} are used to show that the early termination conditions in Steps~\ref{a:tCG:st06},~\ref{a:tCG:st02}, \ref{a:tCG:st04}, and~\ref{a:tCG:st03} in Algorithm~\ref{alg:tCG} do not take effect when $x$ is sufficiently close to $x_*$. The superlinear convergence rate of the proposed method requires step size one to be accepted eventually. To achieve this, we assume the function $F$ to be geodesically strongly convex at $x_*$, as shown in Assumption~\ref{as06}. The definition of geodesically strongly convexity at a point on $\mathcal{M}$ is given in Definition~\ref{d01}.

\begin{definition}[Geodesically strongly convex] \label{d01}
A function $f: \mathcal{M} \rightarrow \mathbb{R}$ is called geodesically $\varsigma$-strongly convex at $x$ if there exists a neighborhood of $x$, denoted by $\tilde{\mathcal{U}}_x$, such that $\Exp$ is a diffeomorphism in $\tilde{\mathcal{U}}_x$ and for any $y \in \tilde{\mathcal{U}}_x$, it holds that
\[
f(y) \geq f(x) + \inner[]{g}{\eta_x} + \frac{\varsigma}{2} \|\eta_x\|_{\F}^2,
\]
where $\eta_x = \Exp_x^{-1}(y)$, $g \in \partial_C f(x)$, and $\partial_C f(x)$ denotes the Riemannian Clarke subgradient of $f$ at $x$ (see in~\cite[Definition~2.3]{GH2015a}).
\end{definition}
\begin{assumption} \label{as06}
The function $F$ is $\varsigma$-geodesically strongly convex at $x_*$. 
\end{assumption}

Since $x_*$ is a local minimizer, it holds that $0_x \in \partial_C F(x_*)$. Therefore, Assumption~\ref{as06} yields
\begin{equation} \label{eq:59}
F(y) \geq F(x_*) + \frac{\varsigma}{2} \|\Exp_{x_*}^{-1}(y)\|_{\F}^2,
\end{equation}
for any $y \in \tilde{\mathcal{U}}_{x_*}$. Moreover, Assumption~\ref{as06} implies Assumption~\ref{as04} as shown in Lemma~\ref{le06}.
\begin{lemma} \label{le06}
Suppose Assumption~\ref{as06} holds, that is, the function $F = f + h$ is $\varsigma$-geodesically strongly convex at $x_*$. Then the linear operator $\mathcal{B}_{x_*}$ is positive definite on $\mathfrak{L}_{x_*}$.
\end{lemma}
\begin{proof}
The proofs rely on the derivation of~\cite[Proposition~3.13]{SAHJV2024}. By Assumption~\ref{as03} and~\cite[Remark~3.12]{SAHJV2024}, we have that \cite[Lemma~3.11 and Proposition~3.13]{SAHJV2024} holds. 
Therefore, the set $\mathcal{M}_{\mathrm{sub}} = \{x \in \Omega_{x_*} \subset \mathcal{M}: \hat{x} = 0\}$ is an embedded submanifold, where $\Omega_{x_*}$ is a sufficiently small neighborhood of $x_*$.
Let $\gamma: (-\varepsilon, \varepsilon) \rightarrow \mathcal{M}_{\mathrm{sub}}$ be a smooth curve on $\mathcal{M}_{\mathrm{sub}}$ such that $\varepsilon > 0$, $\gamma(0) = x_*$, and $\|\gamma'(0)\|_{\F} = 1$. Therefore, $F \circ \gamma$ is smooth on $(- \varepsilon, \varepsilon)$. Using the same derivation as that in~\cite[Proposition~3.13]{SAHJV2024}, we have
\begin{equation} \label{eq:62}
\frac{\mathrm{d}^2}{ \mathrm{d} t^2} F (\gamma(t)) \vert_{t = 0} = u^T \mathcal{B}_{x_*} u,
\end{equation}
where $u$ satisfies $\gamma'(0) = \begin{pmatrix} u \\ 0_{n-j} \end{pmatrix}$. Since $\gamma'(0)$ can be any tangent vector in $\T_x \mathcal{M}_{\mathrm{sub}}$, $u$ can be any vector in $\mathfrak{L}_{x_*}$. Since $\gamma(t)$ is a first-order approximation of $\Exp_{x_*}(t\gamma'(0))$, we have that $\Exp_{x_*}^{-1}(\gamma(t)) = \gamma'(0) t + \mathcal{O}(t^2)$, where $\Exp$ denotes the exponential mapping on $\mathcal{M}$. It follows that
\begin{equation} \label{eq:61}
\lim_{t \rightarrow 0} \frac{ \|\Exp_{x_*}^{-1}(\gamma(t))\|_{\F} }{ t } = 1.
\end{equation}
Therefore, we have
\begin{align}
F(\gamma(t)) =& F(\gamma(0)) + (F \circ \gamma)'(0)t + (F \circ \gamma)''(b_0) t^2 \nonumber \\
=& F(\gamma(0)) + (F \circ \gamma)''(b_0) t^2 \geq F(x_*) + \frac{\varsigma}{2} \|\Exp_{x_*}^{-1}(\gamma(t))\|_{\F}^2, \label{eq:60}
\end{align}
where the first equation follows from Taylor's expansion with $b_0 \in (0, t)$, the second equation follows from that $x_*$ is a minimizer, and the last inequality follows from~\eqref{eq:59}. Inequality~\eqref{eq:60} gives
\[
(F \circ \gamma)''(b_0) \geq \frac{\varsigma}{2} \left( \frac{\|\Exp_{x_*}^{-1}(\gamma(t))\|_{\F}}{t} \right)^2.
\]
Taking $t \rightarrow 0$ and using~\eqref{eq:61} yield $(F \circ \gamma)''(0) \geq \varsigma / 2$. Combining this with~\eqref{eq:62} gives the desired result.
\end{proof}

Note that Assumption~\ref{as06} only requires strong convexity at a point, which can hold for a non-strongly convex function. For example, $h_1:\mathbb{R} \rightarrow \mathbb{R}: x \mapsto |x|$ is not a strongly convex function. But it is strongly convex at $x_* = 0$ by Definition~\ref{d01} since $|y| = h_1(y) \geq h_1(x_*) + (y-x_*)^2 = y^2$ holds for any $y \in [-1, 1]$.

The local convergence analysis is established using three steps. Suppose that $x_k$ is sufficiently close to a local minimizer $x_*$. Then we have that
\begin{enumerate}
	\item the search direction $d(x_k)$ in Step~\ref{a:RNG:st03} of Algorithm~\ref{alg:IRPN} 
    equals the Newton direction $u(x_k)$ in~\eqref{3-3} if~\eqref{eq:09} is solved exactly (Section~\ref{sec01});
	\item Algorithm~\ref{alg:tCG} stops only due to the accuracy condition in Step~\ref{a:tCG:st01} of Algorithm~\ref{alg:tCG} (Section~\ref{sec05}); and
	\item the accuracy condition in Step~\ref{a:tCG:st01} is sufficient for the local superlinear convergence of Algorithm~\ref{alg:IRPN} and the function value of $F$ is sufficiently descent every two steps locally using step size one (Section~\ref{sec06}).
\end{enumerate}

\subsubsection{Reformulation of the Newton equation~\eqref{3-3}} \label{sec01}

It has been shown in Section~\ref{sec04} that if $\bar{B}_{x_k}$ has full column rank, then the Newton equation~\eqref{3-3} is equivalent to~\eqref{eq:10}. The second equation in~\eqref{eq:10} is consistent with the definition of $\hat{d}(x_k)$. Next, we will show that the solution of the first equation in~\eqref{eq:10},
\begin{equation} \label{eq:11}
	[\bar{B}_{x_k} \bar{B}_{x_k}^{\dagger} + t_k (I_{j_k} - \bar{B}_{x_k} \bar{B}_{x_k}^{\dagger}) \mathcal{B}_{x_k} ] \bar{u}(x_k) = \bar{v}(x_k) - t_k (I_{j_k} - \bar{B}_{x_k} \bar{B}_{x_k}^{\dagger}) 
\mathfrak{B}_{x_k}^{(12)} \hat{v}(x_k),
\end{equation}
equals $\bar{d}(x_k)$ if~\eqref{eq:09} is solved exactly and certain reasonable assumptions hold. Specifically, by~\eqref{eq:11}, we have
\begin{align}
& \{I_{j_k} + (I_{j_k} - \bar{B}_{x_k} \bar{B}_{x_k}^{\dagger}) [-I_{j_k} + t_k \mathcal{B}_{x_k} ] \} \bar{u}(x_k) = \bar{v}(x_k) - t_k (I_{j_k} - \bar{B}_{x_k} \bar{B}_{x_k}^{\dagger}) \mathfrak{B}_{x_k}^{(12)} \hat{v}(x_k) \nonumber\\ &\Rightarrow \bar{u}(x_k) = \{I_{j_k} + (I_{j_k} - \bar{B}_{x_k} \bar{B}_{x_k}^{\dagger}) [-I_{j_k} + t_k \mathcal{B}_{x_k}] \}^{-1} [\bar{v}(x_k) - t_k (I_{j_k} - \bar{B}_{x_k} \bar{B}_{x_k}^{\dagger}) \mathfrak{B}_{x_k}^{(12)} \hat{v}(x_k)] \nonumber \\
&\Rightarrow \bar{u}(x_k) = \{I_{j_k} + (I_{j_k} - \bar{B}_{x_k} \bar{B}_{x_k}^{\dagger}) [-I_{j_k} + t_k \mathcal{B}_{x_k}] \}^{-1} \nonumber \\
&\qquad \qquad \{ I_{j_k} + (I_{j_k} - \bar{B}_{x_k} \bar{B}_{x_k}^{\dagger}) [-I_{j_k} + t_k \mathcal{B}_{x_k}] - (I_{j_k} - \bar{B}_{x_k} \bar{B}_{x_k}^{\dagger}) [-I_{j_k} + t_k \mathcal{B}_{x_k}] \} \bar{v}(x_k) \nonumber \\
&\qquad \qquad - \{I_{j_k} + (I_{j_k} - \bar{B}_{x_k} \bar{B}_{x_k}^{\dagger}) [-I_{j_k} + t_k \mathcal{B}_{x_k}] \}^{-1} t (I_{j_k} - \bar{B}_{x_k} \bar{B}_{x_k}^{\dagger}) \mathfrak{B}_{x_k}^{(12)} \hat{v}(x_k) \nonumber \\
&\Rightarrow  \bar{u}(x_k) = \bar{v}(x_k) - \{I_{j_k} + (I_{j_k} - \bar{B}_{x_k} \bar{B}_{x_k}^{\dagger}) [-I_{j_k} + t_k \mathcal{B}_{x_k}] \}^{-1} \nonumber \\
&\qquad \qquad ( I_{j_k} - \bar{B}_{x_k} \bar{B}_{x_k}^{\dagger} ) \{ [-I_{j_k} + t_k \mathcal{B}_{x_k}] \bar{v}(x_k) + t_k \mathfrak{B}_{x_k}^{(12)} \hat{v}(x_k) \} \nonumber \\
&\Rightarrow \bar{u}(x_k) = \bar{v}(x_k) - \{I_{j_k} + (I_{j_k} - \bar{B}_{x_k} \bar{B}_{x_k}^{\dagger}) Y_{x_k} \}^{-1} ( I_{j_k} - \bar{B}_{x_k} \bar{B}_{x_k}^{\dagger} ) \ell_{x_k}, \label{eq:12}
\end{align}
where the second derivation assumes the invertibility of the corresponding matrix and the last derivation is by letting $Y_{x_k}$ denote $-I_{j_k} + t \mathcal{B}_{x_k}$.

Lemma~\ref{le01} is crucial for proving $\bar{u}(x_k) = \bar{v}(x_k) + w(x_k)$, where $w(x_k)$ is obtained by assuming that~\eqref{eq:09} is solved exactly.

\begin{lemma} \label{le01}
Let $Y \in \mathbb{R}^{j \times j}$ and $B \in \mathbb{R}^{j \times m}$ with $m \leq j$. Suppose that $I_j + Y$ is symmetric positive definite over $\mathfrak{L} = \{w \in \mathbb{R}^j \mid B^T w = 0\}$ 
and that $B$ is full column rank. Then it holds that the unique solution of the problem
\begin{equation} \label{eq:01}
\min_{B^T w = 0} \ell^T w + \frac{1}{2} w^T (I_j + Y) w
\end{equation}
is given by
\begin{equation} \label{eq:07}
w_* = - \left[ I_j + (I_j - B B^{\dagger}) Y \right]^{-1} \left[ I_j - B B^{\dagger} \right] \ell.
\end{equation}
\end{lemma}
\begin{proof}
Let $Q \in \mathbb{R}^{j \times {(j - m)}}$ be an orthonormal basis of $\mathfrak{L}$. By the first-order optimality condition, the solution $w_*$ is given by
$
w_* = - Q (Q^T (I_j + Y) Q)^{-1} Q^T \ell = - Q (I_j + Q^T Y Q)^{-1} Q^T \ell.
$
It follows from $Q (I_{j-m} + M Q)^{-1} = (I_j + Q M)^{-1} Q$ that $w_* = - (I_j + Q Q^T Y)^{-1} Q Q^T \ell$, where $M \in \mathbb{R}^{(j-m) \times j}$ is full row rank. The final result follows from $Q Q^T = I_j - B B^{\dagger}$.
\end{proof}
By~\eqref{eq:10},~\eqref{eq:12}, and Lemma~\ref{le01}, sufficient conditions that guarantee $u(x_k) = d(x_k)$ are given in Corollary~\ref{co01}.
\begin{corollary} \label{co01}
	Suppose $\bar{B}_{x_k}$ has full column rank, $\mathcal{B}_{x_k}$ is symmetric positive definite over $\mathfrak{L}_{x_k}$. Then the proximal Newton equation $J(x_k) [u(x_k)] = - v(x_k)$ in~\eqref{3-3} can be computed by
	\[
	u(x_k) = 
 \begin{pmatrix}
	\bar{v}(x_k) + w(x_k) \\
	\hat{v}(x_k)
\end{pmatrix},
	\]
	where $w(x_k) = \argmin_{\bar{B}_{x_k}^T w = 0} \ell_{x_k}^T w + \frac{1}{2} w^T \mathcal{B}_{x_k} w$.
\end{corollary}

\subsubsection{Termination conditions of Algorithm~\ref{alg:tCG}} \label{sec05}

Lemma~\ref{le13} is used in Lemma~\ref{le04}, which implies that Algorithm~\ref{alg:tCG} does not terminate due to the condition in Step~\ref{a:tCG:st04} if $\vartheta$ is sufficiently small and $x$ is in a sufficiently small neighborhood of $x_*$.

\begin{lemma} \label{le13}
Suppose Assumptions~\ref{as02} and \ref{as03} hold. Then there exists a neighborhood of $x_*$, denoted by $\mathcal{V}_1$, and a positive constant $\sigma_{\epsilon}$ such that $B_x$ is a smooth function of $x$ in $\mathcal{V}_1$ and $\sigma_{\min}(\bar{B}_x) > \sigma_{\epsilon}$ for any $x \in \mathcal{V}_1$.
\end{lemma}
\begin{proof}
It has been shown in~\cite{HAG2016VT} that $B_x$ can be chosen such that $B_x$ is a smooth function of $x$ in a small neighborhood of $x_*$, denoted by $\mathcal{V}_1$. Since $\bar{B}_{x_*}$ has full column rank by Assumption~\ref{as03}, it holds that $\sigma_{\min}(\bar{B}_{x_*}) > 0$. By Assumption~\ref{as02}, the support of $x + v$ does not change in $\mathcal{U}_{x_*}$. Therefore, $\bar{B}_x$ is smooth in $\mathcal{V}_1 \cap \mathcal{U}_{x_*}$. It follows that one can shrink $\mathcal{V}_1$ such that $\sigma_{\min}(\bar{B}_x) > \sigma_{\min}(\bar{B}_{x_*})/2$ for any $x \in \mathcal{V}_1$, i.e., setting $\sigma_{\epsilon} = \sigma_{\min}(\bar{B}_{x_*})/2$.
\end{proof}

\begin{lemma} \label{le04}
Suppose Assumptions~\ref{as01}, ~\ref{as02}, \ref{as03}, and~\ref{as04} hold. Then there exists a neighborhood of $x_*$, denoted by $\mathcal{V}_2$, and a positive constant $\chi_{\epsilon}$ such that the smallest eigenvalue of $\mathcal{B}_x$ on $\mathfrak{L}_{x}$ is greater than $\chi_{\epsilon}$ for all $x \in \mathcal{V}_2$. This implies $\mathcal{B}_x$ is positive definite on $\mathfrak{L}_x$ for all $x \in \mathcal{V}_2$.
\end{lemma}
\begin{proof}
By Assumption~\ref{as02}, the support of $x + v$ does not change in $\mathcal{U}_{x_*}$. Combining this with Assumption~\ref{as01} yields that $\mathcal{B}_x$ is a continuous function of $x$ in $\mathcal{U}_{x_*}$. By Lemma~\ref{le13}, we have that $\bar{B}_x$ is a smooth function of $x$ and $\bar{B}_x$ has full column rank in $\mathcal{V}_1$.
Thus, one can choose a smooth function $\bar{Q}: \mathcal{V}_1 \rightarrow \mathbb{R}^{j \times (j - (n - d))}: x \mapsto \bar{Q}_x$ such that $[\bar{B}_x \; \bar{Q}_x]$ is a square full rank matrix and $\bar{Q}_x$ is orthonormal.

By Assumption~\ref{as04}, $\mathcal{B}_{x_*}$ is positive definite on $\mathfrak{L}_{x_*} = \{w \mid \bar{B}_{x_*}^T w = 0\} = \{w \mid w \in \mathrm{span}(\bar{Q}_{x_*})\}$. Let $\chi_{\min}$ denotes the smallest eigenvalue of $\mathcal{B}_{x_*}$ on $\mathfrak{L}_{x_*}$. Therefore, it holds that
\[
c^T \bar{Q}_{x_*}^T \mathcal{B}_{x_*} \bar{Q}_{x_*} c \geq \chi_{\min} \|c\|_{\F}^2,
\]
for any $c \in \mathbb{R}^{j - (n - d)}$. Since $Q_x$ and $\mathcal{B}_{x}$ are continuous with respect to $x$, we have that there exists a neighborhood of $x_*$, denoted by $\mathcal{V}_2 \subseteq \mathcal{V}_1$, such that for any $x \in \mathcal{V}_2$,
\[
c^T \bar{Q}_{x}^T \mathcal{B}_{x} \bar{Q}_{x} c \geq \frac{\chi_{\min}}{2} \|c\|_{\F}^2.
\]
holds for any $c \in \mathbb{R}^{j - (n - d)}$. It follows that the smallest eigenvalue of $\mathcal{B}_{x}$ is greater than $\frac{\chi_{\min}}{2}$ on $\mathfrak{L}_x = \{w \mid \bar{B}_{x} w = 0\} = \{w \mid w \in \mathrm{span}(\bar{Q}_{x})\}$ for any $x \in \mathcal{V}_2$, which completes the proof by setting $\xi_{\epsilon} = \frac{\chi_{\min}}{2}$.
\end{proof}

Lemma~\ref{le12} shows that Algorithm~\ref{alg:tCG} does not terminate due to Step~\ref{a:tCG:st02} or the first condition in Step~\ref{a:tCG:st03} for $\gamma \leq \chi_{\epsilon} / 2$ and $x$ in a small neighborhood of $x_*$.

\begin{lemma} \label{le12}
Suppose Assumptions~\ref{as01}, \ref{as02}, \ref{as03}, and~\ref{as04} hold. If $\tau$ is sufficiently large, then there exists a neighborhood of $x_*$, denoted by $\mathcal{V}_3$, such that for any $x \in \mathcal{V}_3$ and any $\eta \in \T_x \mathcal{M}$, it holds that 
\[
\inner[]{\eta}{ \mathfrak{B}_x \eta} + \tau \|\hat{\eta}\|_{\F}^2 \geq \frac{\chi_{\epsilon}}{2} \|\eta\|_{\F}^2,
\]
where $\chi_{\epsilon}$ is defined in Lemma~\ref{le04}.
\end{lemma}

\begin{proof}
Decompose $\bar{\eta} = \bar{\eta}_1 + \bar{\eta}_2$ such that $\bar{B}_x^T \bar{\eta}_1 = 0$ and $\bar{\eta}_2 \in \mathrm{span}(\bar{B}_x)$. By Lemma~\ref{le13}, we have that $\sigma_{\min}(\bar{B}_x) > \sigma_{\epsilon}$ and $\bar{B}_x$ is a smooth function of $x$ for any $x \in \mathcal{V}_1$. Therefore, $b_0 = \max_{x \in \mathcal{V}_1} \|\bar{B}_x (\bar{B}_x^T \bar{B}_x)^{-1} \hat{B}_x^T\|_2$ is finite. It follows that for any $x \in \mathcal{V}_1$, it holds that
\begin{align} 
\|\bar{\eta}_2\|_{\F} =& \|\bar{B}_x (\bar{B}_x^T \bar{B}_x)^{-1} \bar{B}_x^T \bar{\eta}_2\|_{\F} = \|- \bar{B}_x (\bar{B}_x^T \bar{B}_x)^{-1} \hat{B}_x^T \hat{\eta}\|_{\F} \nonumber \\
\leq& \|- \bar{B}_x (\bar{B}_x^T \bar{B}_x)^{-1} \hat{B}_x^T\|_2 \|\hat{\eta}\|_{\F} \leq b_0 \|\hat{\eta}\|_{\F}, \label{eq:41}
\end{align}
where the first equation follows from $\bar{\eta}_2 \in \mathrm{span}(\bar{B}_x)$ and the second equation follows from $\eta \in \T_x \mathcal{M} \Leftrightarrow B_x^T \eta = 0 \Leftrightarrow \bar{B}_x^T \bar{\eta} + \hat{B}_x^T \hat{\eta} = 0 \Leftrightarrow \bar{B}_x^T \bar{\eta}_2 + \hat{B}_x^T \hat{\eta} = 0$.
For any $x \in \mathcal{V}_1 \cap \mathcal{V}_2$, we have
\begin{align*}
& \inner[]{\eta}{ \mathfrak{B}_x \eta} + \tau \|\hat{\eta}\|_{\F}^2 \\
=& \inner[]{\bar{\eta}_1}{\mathcal{B}_x \bar{\eta}_1} + 2 \inner[]{\bar{\eta}_1}{ \begin{pmatrix} \mathcal{B}_x & \mathfrak{B}_x^{(12)} \end{pmatrix} \begin{pmatrix} \bar{\eta}_2 \\ \hat{\eta} \end{pmatrix}} + \inner[]{\begin{pmatrix} \bar{\eta}_2 \\ \hat{\eta} \end{pmatrix}}{ \mathfrak{B}_x \begin{pmatrix} \bar{\eta}_2 \\ \hat{\eta} \end{pmatrix}} + \frac{\tau}{2} \|\hat{\eta}\|_{\F}^2 + \frac{\tau}{2} \|\hat{\eta}\|_{\F}^2 \\
\geq& \chi_{\epsilon} \|\bar{\eta}_1\|_{\F}^2 - 2 b_1 \|\bar{\eta}_1\|_{\F} \left\|\begin{pmatrix} \bar{\eta}_2 \\ \hat{\eta} \end{pmatrix}\right\|_{\F} - b_1 \left\|\begin{pmatrix} \bar{\eta}_2 \\ \hat{\eta} \end{pmatrix}\right\|_{\F}^2 + \frac{\tau}{2} \|\hat{\eta}\|_{\F}^2 + \frac{\tau}{2 b_0} \|\bar{\eta}_2\|_{\F}^2 \\
\geq& \chi_{\epsilon} \|\bar{\eta}_1\|_{\F}^2 - \left( \left( \sqrt{\frac{\chi_{\epsilon}}{2}}\|\bar{\eta}_1\|_{\F}\right)^2 + \left( \sqrt{\frac{2}{\chi_{\epsilon}}} b_1  \left\|\begin{pmatrix} \bar{\eta}_2 \\ \hat{\eta} \end{pmatrix}\right\|_{\F} \right)^2 \right) \\
&\qquad \qquad - b_1 \left\|\begin{pmatrix} \bar{\eta}_2 \\ \hat{\eta} \end{pmatrix}\right\|_{\F}^2 + \frac{\tau}{2} \|\hat{\eta}\|_{\F}^2 + \frac{\tau}{2 b_0} \|\bar{\eta}_2\|_{\F}^2 \\
\geq& \frac{\chi_{\epsilon}}{2} \| \bar{\eta}_1 \|_{\F}^2 - \left( \frac{2}{\chi_{\epsilon}} b_1^2 + b_1 \right) (b_0^2 + 1) \|\hat{\eta}\|_{\F}^2 + \frac{\tau}{2} \|\hat{\eta}\|_{\F}^2 + \frac{\tau}{2 b_0} \|\bar{\eta}_2\|_{\F}^2,
\end{align*}
where the first inequality follows from Lemma~\ref{le04}, the Cauchy-Schwarz inequality, $b_1 = \max_{x \in \mathcal{V}_1 \cap \mathcal{V}_2} \|\mathfrak{B}_x\|_2$, and~\eqref{eq:41}, and the second inequality follows from the inequality $2 a b \leq a^2 + b^2$, and the last inequality follows from~\eqref{eq:41}. 
If $\frac{\tau}{2 b_0} \geq \frac{\chi_{\epsilon}}{2}$ and $\frac{\tau}{2} - \left( \frac{2}{\chi_{\epsilon}} b_1^2 + b_1 \right) (b_0^2 + 1) \geq \frac{\chi_{\epsilon}}{2}$, 
then it holds that $\inner[]{\eta}{ \mathfrak{B}_x \eta} + \tau \|\hat{\eta}\|_{\F}^2 \geq \frac{\chi_{\epsilon}}{2} (\|\bar{\eta}_1\|_{\F}^2 + \|\bar{\eta}_2\|_{\F}^2 + \|\hat{\eta}\|_{\F}^2) = \frac{\chi_{\epsilon}}{2} \|\eta\|_{\F}^2$.
\end{proof}

Lemma~\ref{le05} states that $\|x - x_*\|$ can be bounded by the norm of $v(x)$ and is used to show that $\mathrm{sign}(\bar{x}_*)$ is the same as $\mathrm{sign}(\bar{x} + \bar{v})$ if $x$ is close to $x_*$. This result is used in Lemma~\ref{le14}.
\begin{lemma} \label{le05}
Suppose Assumptions~\ref{as01}, \ref{as02}, \ref{as03} and~\ref{as04} hold. Then there exists a neighborhood of $x_*$, denoted by $\mathcal{V}_4$, two positive constants $C_U$ and $C_L$ such that inequality
\[
C_L \| v(x) \|_{\F} \leq \| x - x_*\|_{\F} \leq C_U \| v(x) \|_{\F}
\]
holds for all $x \in \mathcal{V}_4$.
\end{lemma}
\begin{proof}
The proof relies on the analysis in~\cite[Theorem~3.9]{SAHJV2024}. By Assumption~\ref{as04} and \cite[Proposition~3.13]{SAHJV2024}, the matrix $J(x_*)$ defined as~\eqref{eq:Jsol} is nonsingular. Combining this with Assumptions~\ref{as02} and~\ref{as03} yields that the analysis in~\cite[Theorem~3.9]{SAHJV2024} hold. It follows that there exists a neighborhood of $x_*$, denoted by $\tilde{\mathcal{V}}_4$, such that
\begin{align} \label{eq:37}
	v(x) - v(x_*) - J(x) (x - x_*) = \smallO(\|x - x_*\|_{\F})
\end{align}
holds for all $x \in \tilde{\mathcal{V}}_4$, where $\smallO(t) \in \mathbb{R}^n$ denotes a higher order of $t$ as $t \rightarrow 0$ and $J(x)$ varies continuously with $x \in \tilde{\mathcal{V}}_2$. The final result follows from~\eqref{eq:37}, the continuity of $J(x)$ for $x \in \tilde{\mathcal{V}}_2$, and the nonsingularity of $J(x_*)$.

\end{proof}

Lemma~\ref{le14} gives a new interpretation of the direction $u(x)$ when $x$ is close to a minimizer $x_*$. In~\cite[Section~5.1]{SAHJV2024}, it is shown that a naive generalization of the Euclidean proximal Newton method does not converge superlinearly. Lemma~\ref{le14} implies that one can use ManPG to identify the locations of zero entries and then fix those entries in $v(x)$. The superlinear search direction is therefore computed by~\eqref{eq:56}. Note that the second-order term in~\eqref{eq:56} is not equal to $\Hess f(x)$ in general since the Lagrange multiplier $\lambda$ in the Weingarten map may not be equal to that in $\Hess f(x)$. However, the left upper block matrix $\mathcal{B}_x$ of $\mathfrak{B}_x$ is close to the Hessian of $f$ over the submanifold of $\mathcal{M}$, defined by $\mathcal{M}_{\mathrm{sub}} = \{x \in \Omega_{x_*} \subset \mathcal{M}: \hat{x} = 0\}$, where $\Omega_{x_*}$ is a sufficiently small neighborhood of $x_*$, see details in the proofs of~\cite[Proposition~3.13]{SAHJV2024}. It follows that $w(x)$ can be viewed as a Newton direction of $f$ over the manifold $\mathcal{M}_{\mathrm{sub}}$. 

Lemma~\ref{le14} is used in Lemma~\ref{le15}.

\begin{lemma} \label{le14}
Suppose Assumptions~\ref{as01}, \ref{as02}, \ref{as03} and~\ref{as04} hold. Then there exists a neighborhood of $x_*$, denoted by $\mathcal{V}_5$, such that
\begin{equation} \label{eq:56}
u(x) = d(x) = \argmin_{u \in \T_x \mathcal{M}, \hat{u} = \hat{v}(x)} G_x(u) = \frac{1}{2} \inner[]{u}{ \mathfrak{B}_x u} + \nabla f(x)^T u + \frac{\tau}{2} \|\hat{u}\|_{\F}^2 + \mu \|x + u\|_1
\end{equation}
holds for any $x \in \mathcal{V}_5$, where $d(x)$ is obtained by assuming~\eqref{eq:09} is solved exactly.
\end{lemma}
\begin{proof}
By Assumption~\ref{as02}, $x+v$ has the same support as $x_*$. Therefore, $\bar{x} + \bar{v}$ and $\bar{x}_*$ correspond to the same positions in $x + v$ and $x_*$ respectively. Since any entry in $\bar{x}_*$ is nonzero and Lemma~\ref{le05} holds, there exists a neighborhood of $x_*$, $\tilde{\mathcal{V}}_5 \subseteq \mathcal{U}_{x_*}$, such that for any $x \in \tilde{\mathcal{V}}_5$, it holds that
\begin{equation} \label{eq:42}
\mathrm{sign}(\bar{x} + \bar{v}) = \mathrm{sign}(x_*).
\end{equation}
It follows that for any $x \in \tilde{\mathcal{V}}_5$, Equation~\eqref{3-1} yields
$
    \bar{x} + \bar{v} + t \mu \mathrm{sign}(\bar{x} + \bar{v}) = \bar{x} - t ( \overline{\nabla f(x)} + \bar{B}_x \lambda_x )
$
which gives
\begin{equation} \label{eq:45}
\overline{\nabla f(x)} + \mu \mathrm{sign}(\bar{x}_*) + \frac{1}{t} \bar{v} = \bar{B}_x \lambda_x.
\end{equation}
Consider Problem~\eqref{eq:09}. We have
\begin{align}
\argmin_{w \in \mathfrak{L}_x}& \inner[]{\ell_{x}}{w} + \frac{1}{2} \inner[]{w}{\mathcal{B}_x w} \nonumber \\
= \argmin_{w \in \mathfrak{L}_x}& \inner[]{ \frac{1}{t} ( -I_{j} + t \mathcal{B}_{x} ) \bar{v}(x) + \mathfrak{B}_x^{(12)} \hat{v}(x) }{w} + \frac{1}{2} \inner[]{w}{\mathcal{B}_x w} \nonumber \\
= \argmin_{w \in \mathfrak{L}_x}& \frac{1}{2} \inner[]{w}{\mathcal{B}_x w} + \inner[]{w}{\mathcal{B}_x \bar{v}(x)} + \inner[]{w}{ - \frac{1}{t} \bar{v}(x) + \mathfrak{B}_x^{(12)} \hat{v}(x) } \nonumber \\
= \argmin_{w \in \mathfrak{L}_x}& \frac{1}{2} \inner[]{w}{\mathcal{B}_x w} + \inner[]{w}{\mathcal{B}_x \bar{v}(x)} + \inner[]{w}{ \overline{\nabla f(x)} + \mu \mathrm{sign}(\bar{x}_*) + \mathfrak{B}_x^{(12)} \hat{v}(x) } \nonumber \\
& \qquad \qquad \qquad \qquad \qquad\qquad \qquad \qquad \qquad  \hbox{ (by~\eqref{eq:45} and $w \in \mathfrak{L}_x$)} \nonumber \\
= \argmin_{w \in \mathfrak{L}_x}& \frac{1}{2} \inner[]{\bar{v}(x) + w}{\mathcal{B}_x (\bar{v}(x) + w)} + \inner[]{\bar{v}(x) + w}{ \overline{\nabla f(x)} + \mu \mathrm{sign}(\bar{x}_*) + \mathfrak{B}_x^{(12)} \hat{v}(x) } \nonumber \\
=& \left( \argmin_{\bar{u} - \bar{v}(x) \in \mathfrak{L}_x} \frac{1}{2} \inner[]{\bar{u}}{\mathcal{B}_x \bar{u}} + \inner[]{\bar{u}}{ \overline{\nabla f(x)} + \mu \mathrm{sign}(\bar{x}_*) + \mathfrak{B}_x^{(12)} \hat{v}(x) } \right) - \bar{v}(x), \label{eq:46}
\end{align}
where the second to the last equation follows from that adding a constant to the objective function does not change the minimizer and the last equation follows from letting $\bar{u} = \bar{v}(x) + w$. It follows from~\eqref{eq:46} that
\begin{align}
&u(x) = \argmin_{u \in \T_x \mathcal{M}, \hat{u} = \hat{v}(x)}  \frac{1}{2} \inner[]{\bar{u}}{\mathcal{B}_x \bar{u}} + \inner[]{\bar{u}}{ \overline{\nabla f(x)} + \mu \mathrm{sign}(\bar{x}_*) + \mathfrak{B}_x^{(12)} \hat{v}(x) } \nonumber \\
=& \argmin_{u \in \T_x \mathcal{M}, \hat{u} = \hat{v}(x)} \frac{1}{2} \inner[]{\begin{pmatrix} \bar{u} \\ \hat{u} \end{pmatrix}}{ \begin{pmatrix} \mathcal{B}_x & \mathfrak{B}_x^{(12)} \\ \mathfrak{B}_x^{(21)} & 0_{(n-j) \times (n-j)} \end{pmatrix} \begin{pmatrix} \bar{u} \\ \hat{u} \end{pmatrix}} + \inner[]{\bar{u}}{ \overline{\nabla f(x)} + \mu \mathrm{sign}(\bar{x}_*) } \nonumber \\
=& \argmin_{u \in \T_x \mathcal{M}, \hat{u} = \hat{v}(x)} \frac{1}{2} \inner[]{u}{ \mathfrak{B}_x u} + \inner[]{\bar{u}}{ \overline{\nabla f(x)} + \mu \mathrm{sign}(\bar{x}_*) } \nonumber \\
=& \argmin_{u \in \T_x \mathcal{M}, \hat{u} = \hat{v}(x)} \frac{1}{2} \inner[]{u}{ \mathfrak{B}_x u} + \inner[]{\bar{u}}{ \overline{\nabla f(x)}} + \mu \inner[]{\bar{u}}{\mathrm{sign}(\bar{x} + \bar{u}) } \qquad \hbox{ (by~\eqref{eq:42})} \nonumber \\
=& \argmin_{u \in \T_x \mathcal{M}, \hat{u} = \hat{v}(x)} \frac{1}{2} \inner[]{u}{ \mathfrak{B}_x u} + \nabla f(x)^T u + \frac{\tau}{2} \|\hat{u}\|_{\F}^2 + \mu \|x + u\|_1, \nonumber
\end{align}
where the third and the last equations follow from that adding a constant to the objective function does not change the minimizer.
\end{proof}

Lemmas~\ref{le15} and~\ref{le16} show that Algorithm~\ref{alg:tCG} does not terminate due to Step~\ref{a:tCG:st06} and the second condition in Step~\ref{a:tCG:st03} if $x$ is close to $x_*$. 

\begin{lemma} \label{le15}
Suppose Assumptions~\ref{as01}, \ref{as02}, \ref{as03} and~\ref{as04} hold. Then it holds that  
\[
G_x(d_{i+1}) \leq G_x(v(x))
\]
for any $x \in \mathcal{V}_5$ and any $i$, where $\mathcal{V}_5$ is defined in Lemma~\ref{le14}.
\end{lemma}
\begin{proof}
Since the truncated conjugate gradient method is used to optimize Problem~\eqref{eq:09} and the conjugate gradient method is a descent algorithm, it follows that the objective function in~\eqref{eq:09} decreases at every iteration of 
Algorithm~\ref{alg:tCG}. By the derivation of Lemma~\ref{le14}, the objective function in~\eqref{eq:09} and $G_x$ are only different up to a constant and a variable substitution. Therefore, the value of $G_x$ also decreases. It follows that $G_x(d_{i+1}) \leq G_x(d_0) = G_x(v(x))$.
\end{proof}

\begin{lemma} \label{le16}
Suppose Assumptions~\ref{as01}, \ref{as02}, \ref{as03} and~\ref{as04} hold. If $t$ is sufficiently small, then it holds that 
\[
G_x(v(x)) \leq G_x(0)
\]
holds for any $x \in \mathcal{M}$. Therefore, $G_x(d_{i+1}) \leq G_x(0)$ for any $x \in \mathcal{V}_5$ by Lemma~\ref{le15}.
\end{lemma}
\begin{proof}
By the definition of $v(x)$, we have that
\begin{equation*} 
\inner[]{\nabla f(x)}{v(x)} + \frac{1}{2 t} \|v(x)\|_{\F}^2 + \mu \|x + v(x)\|_1 \leq \mu \|x\|_1.
\end{equation*}
We have
\begin{align}
&G_x(v(x)) - G_x(0) \nonumber \\
=& \frac{1}{2} \inner[]{v(x)}{ \mathfrak{B}_x v(x)} + \inner[]{\nabla f(x)}{v(x)} + \frac{\tau}{2} \|\hat{v}(x)\|_{\F}^2 + \mu \|x + v(x)\|_1 - \mu \|x\|_1 \nonumber\\
\leq& \frac{1}{2} \inner[]{v(x)}{ \mathfrak{B}_x v(x)} + \frac{\tau}{2} \|\hat{v}(x)\|_{\F}^2 - \frac{1}{2 t} \|v(x)\|_{\F}^2 
\leq \frac{1}{2} \left( \| \mathfrak{B}_x \|_2 + \tau - \frac{1}{t} \right) \|v(x)\|_{\F}^2 
\label{eq:49}
\end{align}
Therefore, if $t \leq \frac{1}{\tau + \max_{x \in \mathcal{M}} \|\mathfrak{B}_x\|_2}$, then the right term of~\eqref{eq:49} is negative, which completes the proof.
\end{proof}

By Lemmas~\ref{le04}, \ref{le12}, \ref{le15}, and~\ref{le16}, it is concluded that Algorithm~\ref{alg:tCG} only terminates due to the accurate condition in Step~\ref{a:tCG:st01} if $x$ is sufficiently close to $x_*$, $t$ is sufficiently small, and Assumptions~\ref{as01}, \ref{as02}, \ref{as03} and~\ref{as04} hold.

\subsubsection{Superlinear convergence analysis} \label{sec06}

If the step size one is used, then the accuracy termination condition in Algorithm~\ref{alg:tCG} ensures that the convergence rate of Algorithm~\ref{alg:IRPN} is superlinear.
\begin{theorem} \label{th02}
Suppose Assumptions~\ref{as01}, \ref{as02}, \ref{as03} and~\ref{as04} hold and $t$ is sufficiently small such that the result of Lemma~\ref{le16} holds. Then there exists a neighborhood of $x_*$, denoted by $\mathcal{V}_6$, such that if the step size one is used, then the convergence rate is $\min(1 + \theta, 2)$, i.e., $\|R_x(d(x)) - x_*\|_{\F} \leq C_{\mathrm{up}} \|x - x_*\|_{\F}^{\min(1 + \theta, 2)}$ holds for any $x \in \mathcal{V}_6$ and a constant $C_{\mathrm{up}} > 0$.
\end{theorem}
\begin{proof}
Let $\tilde{\ell}_{x,t}$ denote $P_x(\ell_{x,t})$, $\tilde{\mathcal{B}}_x$ denote $P_x \mathcal{B}_x P_x$, and $\tilde{\mathcal{V}}_5$ denote $\cap_{i = 1}^5 \mathcal{V}_i$. Consider $x \in \tilde{\mathcal{V}}_5$, we have
\begin{align} \label{eq:57}
\|\tilde{\ell}_{x,t}\|_{\F} \leq \|\ell_{x,t}\|_{\F} \leq b_0 \|v(x)\|_{\F} \leq \frac{b_0}{C_L} \|x - x_*\|_{\F},
\end{align}
where $b_0 = \max( \max_{x \in \mathcal{M}} \|(-I_j + t \mathcal{B}_x) / t\|_2, \max_{x \in \mathcal{M}} \|\mathfrak{B}_x^{(12)}\|_2 )$ and the third inequality follows from Lemma~\ref{le05}. Therefore, it holds that for $x \in \tilde{\mathcal{V}}_5 \cap \tilde{\mathcal{V}}$,
\begin{align}
&\|x + d(x) - x_*\|_{\F} \leq \|x + u(x) - x_*\|_{\F} + \|u(x) - d(x)\|_{\F} \nonumber \\
\leq C_{\mathrm{sup}} \|x - x_*\|_{\F}^2& + \|w(x) - w_x^*\|_{\F} 
\leq C_{\mathrm{sup}} \|x - x_*\|_{\F}^2 + \|\tilde{\mathcal{B}}_x^{-1}\|_{2} \|\tilde{\mathcal{B}}_x w(x) - \tilde{\ell}_{x,t}\|_{\F} \nonumber \\
\leq C_{\mathrm{sup}} \|x - x_*\|_{\F}^2& + \|\tilde{\mathcal{B}}_x^{-1}\|_{2} \| \|\tilde{\ell}_{x,t}\|_{\F}^{(1 + \theta)} \leq C_{\mathrm{sup}} \|x - x_*\|_{\F}^2 + b_1 \|x - x_*\|_{\F}^{1 + \theta}\nonumber \\
\leq& \max( C_{\mathrm{sup}}, b_1 ) \|x - x_*\|_{\F}^{\min(1 + \theta, 2)}, \label{eq:58}
\end{align}
where $\tilde{\mathcal{V}}$ denotes a neighborhood of $x_*$ such that $u(x)$ is the quadratic convergence direction, i.e., the Riemannian proximal Newton direction~\cite{SAHJV2024arxiv}, $C_{\mathrm{sup}} > 0$ is a constant, $w_x^*$ denotes the exact solution of~\eqref{eq:09}, the third inequality uses $\tilde{\mathcal{B}}_x w_x^* = \tilde{\ell}_{x,t}$ and Lemma~\ref{le04}, the fourth inequality follows from Step~\ref{a:tCG:st01} in Algorithm~\ref{alg:tCG} with the assumption $\|\tilde{\ell}_{x,t}\|_{\F}^{\theta} \leq \kappa$, and the fifth inequality follows from $b_1 = \frac{b_0^{1 + \theta}}{C_L^{1 + \theta}} \max_{x \in \mathcal{V}_2} \|\tilde{\mathcal{B}}_x^{-1}\|_2$ and~\eqref{eq:57}. 
Note that the assumption $\|\tilde{\ell}_{x,t}\|_{\F}^{\theta} \leq \kappa$ can be made since $\kappa$ is a constant, $\|\tilde{\ell}_{x,t}\|_{\F} \leq \frac{b_0}{C_L} \|x - x_*\|_{\F}$ by~\eqref{eq:57}, and the neighborhood $\tilde{\mathcal{V}}$ is chosen sufficiently small.
Equation~\eqref{eq:58} yields that $\|d(x)\|_{\F} \leq (1 + \max(C_{\mathrm{sup}}, b_1 )) \|x - x_*\|_{\F}$ with $x \in \mathcal{V}_6:= \tilde{V}_5 \cap \tilde{\mathcal{V}} \cap \{x\in \M: \|x-x_*\|_{\F}<1\}$. It follows that
\begin{align*}
&\|R_x(d(x)) - x_*\|_{\F} \leq \|R_x(d(x)) - x - d(x)\|_{\F} + \|x + d(x) - x_*\|_{\F} \\
\leq& C_{R_2} \|d(x)\|_{\F}^2 + \max( C_{\mathrm{sup}}, b_1 ) \|x - x_*\|_{\F}^{\min(1 + \theta, 2)} \\
\leq& C_{R_2} (1 + \max(C_{\mathrm{sup}}, b_1 ))^2 \|x - x_*\|_{\F}^2 + \max( C_{\mathrm{sup}}, b_1 ) \|x - x_*\|_{\F}^{\min(1 + \theta, 2)},
\end{align*}
which implies the superlinear convergence rate on the order of $\min(1 + \theta, 2)$.
\end{proof}
However, the step size one with the proximal Newton direction $u(x)$ does not guarantee the function value $F$ to decrease, which is different from the Euclidean case and the smooth Riemannian case. Example~\ref{ex01} gives such an example. Note that this example relies on the nonsmoothness of the problem and the curvature in retraction.
\begin{example} \label{ex01}
Consider the function $F: \mathbb{R}^2\rightarrow \mathbb{R}: (x_1, x_2)^T \mapsto x_1^2 - 3 x_1 + 1 + x_2^2 + |x_1| + |x_2|$. Therefore, this objective function is in the form of~\eqref{eq:F} with $f(x) = x_1^2 - 3 x_1 + 1 + x_2^2$, $\mu = 1$, and $\mathcal{M} = \mathbb{R}^2$. It is easy to verify that the unique minimizer of $F$ is $x_* = (1, 0)^T$. Let $x = (1 + \epsilon, 0)^T$ with $|\epsilon| < 1$. Since $f(x)$ is a quadratic function, the proximal Newton direction $u(x)$ points to $x_*$ and therefore is given by $u(x) = x_* - x = (-\epsilon, 0)^T$. We use the retraction defined by $R:\T \mathcal{M} \rightarrow \mathcal{M}: \eta_x \mapsto x + \eta_x + \begin{pmatrix} 0 \\ 2 \eta_x^T \eta_x \end{pmatrix}$. It follows that
\begin{align}
&F(R_x(u(x))) - F(x) = F(1, 2 \epsilon^2) - F(1 + \epsilon, 0) \nonumber \\
=& 4 \epsilon^4 + |2 \epsilon^2| - ( (1 + \epsilon)^2 - 3(1+\epsilon) + 1 + |1 + \epsilon|)
= 4 \epsilon^4 + \epsilon^2 > 0. \label{eq:68}
\end{align}
Since $\epsilon$ can be arbitrarily small, Inequality~\eqref{eq:68} implies that there does not exist a neighborhood of $x_*$ such that for any $x$ in the neighborhood, the proximal Newton direction $u(x)$ with step size one ensures a decrease in the objective function. However, if the retraction $R_x(\eta_x) = x + \eta_x$ is used, i.e., zero curvature, then it has been proven in~\cite{LSS2014} that the step size one can be used and convergence rate is quadratic. 
\end{example}
Although using step size one in an iteration can not guarantee descent, using step size one for two consecutive iterations near $x_*$ can, as proved in Proposition~\ref{p02}.
\begin{proposition} \label{p02}
Suppose that Assumptions~\ref{as01}, \ref{as02}, \ref{as03} and~\ref{as06} hold and that there exists a neighborhood of $x_*$, denoted by $\mathcal{V}_9$, such that for any $x \in \mathcal{V}_9$, it holds that $\|R_x(d(x)) - x_*\|_{\F} \leq C_{\mathrm{up}} \|x - x_*\|_{\F}^{\varkappa}$ for a $\varkappa > \sqrt{2}$ and $R_x(d(x)) \in \mathcal{V}_9$. Then there exists a neighborhood of $x_*$, denoted by $\mathcal{V}_{10}$, and a constant $\rho_1> 0$ such that for any $x \in \mathcal{V}_{10}$, it holds that 
\[
F(x_{++}) \leq F(x) - \rho_1 \|v(x)\|_{\F}^2,
\]
where $x_+ = R_x(d(x))$ and $x_{++} = R_{x_+}(d(x_+))$.
\end{proposition}
\begin{proof}
By Assumption~\ref{as01}, $F$ is a Lipschitz continuous function. Let $L_F$ denote its Lipschitz constant. It follows that for any $x \in \mathcal{V}_9$, it holds that
\begin{align} \label{eq:64}
F(x_{++}) - F(x_*) \leq L_F \|x_{++} - x_*\|_{\F} \leq L_F C_{\mathrm{up}} \|x_+ - x_*\|_{\F}^{\varkappa} \leq L_F C_{\mathrm{up}}^{1 + \varkappa} \|x - x_*\|_{\F}^{(\varkappa^2)}.
\end{align}
Using~\eqref{eq:27} for the exponential mapping yields
\begin{equation} \label{eq:63}
\|\Exp_{x_*}^{-1}(x)\|_{\F} \geq \|x - x_*\|_{\F} - C_{\Exp_2} \|\Exp_{x_*}^{-1}(x)\|_{\F}^2 \geq \frac{1}{2} \|x - x_*\|_{\F},
\end{equation}
where $C_{\Exp_2}$ is a constant and the second inequality assumes $x$ to be sufficiently close to $x_*$. By~\eqref{eq:63} and~\eqref{eq:59}, we have
\begin{equation} \label{eq:66}
F(x) \geq F(x_*) + \frac{\varsigma}{2} \|\Exp_{x_*}^{-1}(x)\|_{\F}^2 \geq F(x_*) + \frac{\varsigma}{8} \|x - x_*\|_{\F}^2.
\end{equation}
Since $\varkappa^2 > 2$, there exists a neighborhood of $x_*$, denoted by $\tilde{\mathcal{V}}_{10}$, such that 
\begin{equation} \label{eq:65}
L_F C_{\mathrm{up}}^{1 + \varkappa}  \|x - x_*\|_{\F}^{(\varkappa^2)} \leq \frac{\varsigma}{16} \|x - x_*\|^2.
\end{equation}
Using~\eqref{eq:64}, \eqref{eq:66} and~\eqref{eq:65} yields
\[
F(x_{++}) \leq F(x_*) + L_F C_{\mathrm{up}}^{1 + \varkappa} \|x - x_*\|_{\F}^{(\varkappa^2)} \leq F(x) - \frac{\varsigma}{16} \|x - x_*\|_{\F}^2.
\]
By Lemma~\ref{le05}, we have $F(x_{++}) \leq F(x) - \frac{\varsigma C_L^2}{16} \|v(x)\|_{\F}^2$ for any $x$ in a sufficiently small neighborhood of $x_*$.
\end{proof}
Finally, the global and local superlinear convergence of Algorithm~\ref{alg:IRPN} follows from Theorem~\ref{th03}, Theorem~\ref{th02}, and Proposition~\ref{p02}, and stated in Theorem~\ref{th04}.
\begin{theorem} \label{th04}
Suppose that Assumptions~\ref{as01}, \ref{as02}, \ref{as03} and~\ref{as06} hold and that $t_{\min}$ is sufficiently small such that the results of Lemmas~\ref{le11} and~\ref{le16} hold. Let $\{x_k\}$ denote the sequence generated by Algorithm~\ref{alg:IRPN}. 
Then $\liminf_{k \rightarrow \infty} \|v(x_k)\|_{\F} = 0$ hold. Moreover, there exists a neighborhood of $x_*$, denoted by $\mathcal{V}_{10}$, such that if an iterate $x_k$ is in $\mathcal{V}_{10}$, then the sequence $\{x_k\}$ generated by Algorithm~\ref{alg:IRPN} converges to $x_*$ and the local convergence rate is $\min(1+\theta, 2)$.
\end{theorem}
\begin{proof}
The global convergence $\liminf_{k \rightarrow \infty} \|v(x_k)\|_{\F} = 0$ follows from Theorem~\ref{th03}. By Step~\ref{a:RNG:st17} of Algorithm~\ref{alg:IRPN}, Step~\ref{a:tCG:st06} of Algorithm~\ref{alg:tCG}, and Lemma~\ref{le16}, we have that $t$ is sufficiently small for Lemma~\ref{le16} when $k$ is sufficiently large. The local convergence rate of $\min(1+\theta, 2)$ follows from Theorem~\ref{th02}, Proposition~\ref{p02}, and Steps~\ref{a:RNG:st07} to~\ref{a:RNG:st08} of Algorithm~\ref{alg:IRPN}.
\end{proof}

\section{Numerical Experiments} \label{sec12}

In this section, we compare the proposed method with the existing proximal gradient-type methods, i.e., ManPG~\cite{CMSZ2019}, ManPG-Ada~\cite{CMSZ2019}, ManPQN~\cite{WY2023}, RPN-H~\cite[Algorithm~4.1]{SAHJV2024} and IAManPG~\cite{HWGV2022}. The tested problems, support estimation, parameter setting, and testing environment are given in Section~\ref{sec10}.  In Section~\ref{sec11}, we first combine RPN-CG with ManPG-Ada and propose a hybrid method named RPN-CGH. The robustness of the three superlinearly converging methods, RPN-H, RPN-CG, and RPN-CGH, are compared in Section~\ref{sec11}. 
Next, the efficiency of RPN-CG and RPN-CGH is compared to ManPG, ManPG-Ada, ManPQN, and RPN-H using the sparse principal component analysis (sparse PCA) and the compressed modes (CM) problems respectively in Section~\ref{sec:spca} and Section~\ref{sec:cm}. Finally, we combine RPN-CG with IAManPG and obtained a hybrid method also called RPN-CGH, the efficiency of RPN-CG and RPN-CGH is compared to ManPG and IAManPG using community detection problems in Section~\ref{subsec:cd}.



\subsection{Tested problems, support estimation, parameter setting, and testing environment} \label{sec10}

\paragraph{Tested problems:} We use problems from a sparse PCA model and discretized compressed modes for numerical tests. The sparse PCA model in the form of
\begin{equation}\label{eq:spca}
    \min_{X\in \mathrm{St}(r, n)} -\trace(X^{\T}A^{\T}AX) + \mu \|X\|_1
\end{equation}
is considered, where $A\in \R^{m\times n}$ is the data matrix. The formulation~\eqref{eq:spca} is a penalized version of the ScoTLASS model proposed by Jolliffe et al.~\cite{JoTrUd2003a} and has been used in~\cite{CMSZ2019,HuaWei2019,WY2023,SAHJV2024} for benchmarking Riemannian proximal gradient-type methods.

Two types of data matrices $A$ are used in this paper.
\begin{enumerate}
\item {\bf Random data.} The entries in the data matrix $A$ are drawn from the standard normal distribution~$\mathcal{N}(0, 1)$.
\item {\bf Synthetic data.}  We first repeat the five principal components (shown in Figure~\ref{fig:PCs})  $m/5$ times to obtain an $m$-by-$n$ noise-free matrix. Then the data matrix $A$ is created by further adding a random noise matrix, where each entry of the noise matrix is drawn from~$\mathcal{N}(0, 0.64)$. This synthetic data generation is inspired by the work in~\cite{SCLEE2018} and has been used in~\cite{HuaWei2019,SAHJV2024}.
\end{enumerate}
\begin{figure}[ht!]
\centering
\includegraphics[width=0.5\textwidth]{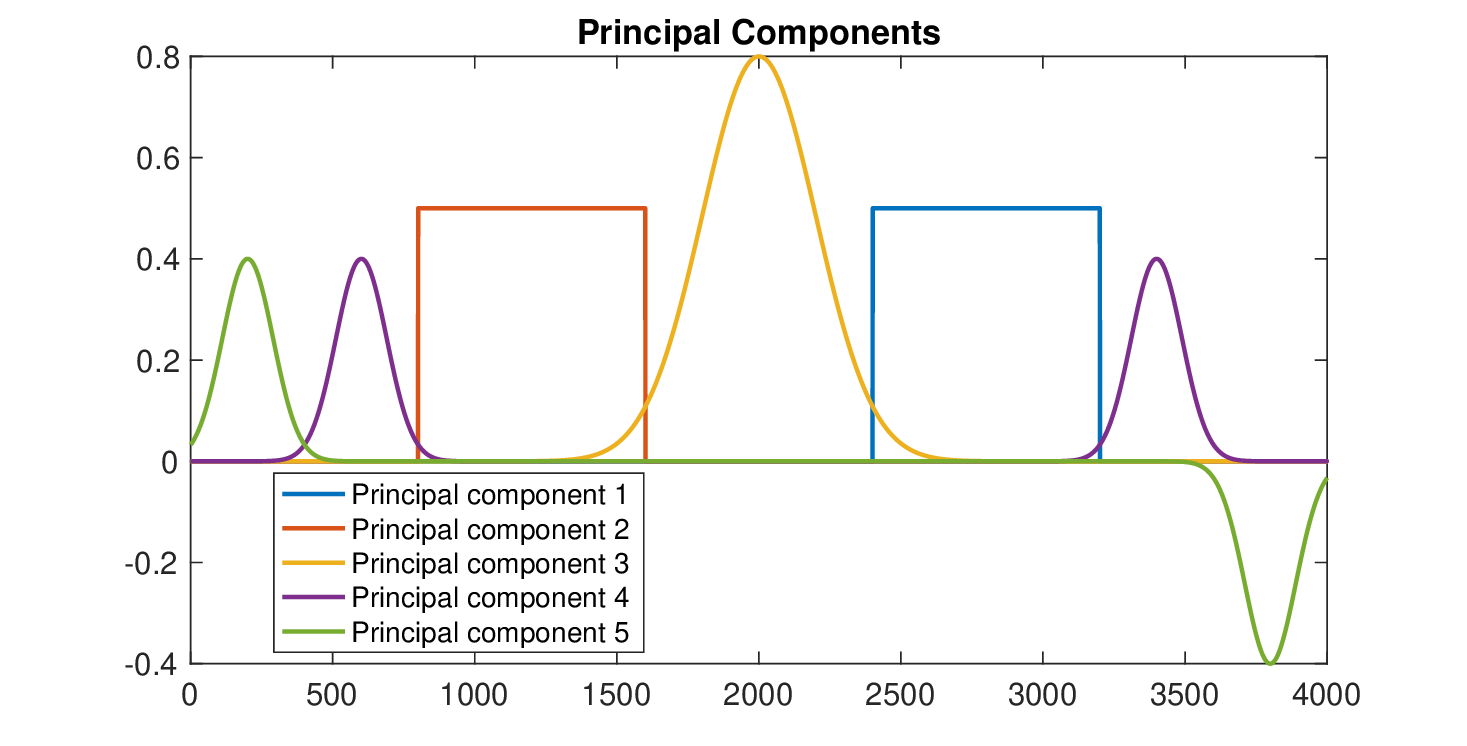}
\caption{
The five principal components used in the synthetic data.
}
\label{fig:PCs}
\end{figure}

The compressed modes (CM) problem aims to seek sparse solution of the independent-particle Schr{\"o}dinger equation. 
After proper discretization, CM problem~\cite{OLCO2013} can be formulated as
\begin{equation}\label{eq:cm}
    \min_{X\in \mathrm{St}(r,n)} \trace(X^{\T}HX) + \mu \|X\|_1,
\end{equation}
where $H \in \R^{n\times n}$ denotes the discretized Schr{\"o}dinger operator.

The community detection (CD) problem is formulated as the optimization problem
\begin{equation}\label{eq:cd}
    \min_{X\in \mathcal{F}_{\mathbf{1}_n}} -\trace(X^{\T}M X) + \mu \|X\|_1,
\end{equation}
where $M \in \mathbb{R}^{n\times n}$ is the modularity matrix, $\mathcal{F}_{v} = \{ X\in \mathbb{R}^{n \times r}: X^{\T}X = \I_{r}, v \in \mathrm{span}(X)\} $ with $v = \mathbf{1}_n$ denotes the vector with all entries being one, $\mathrm{span}(X)$ denotes the column space of $X$. It follows from~\cite[Theorem~1]{HWGV2022} that $\mathcal{F}_v$, where $v\in \mathbb{R}^n$ is a vector with all entries being positive, is a compact embedded submanifold of $\mathbb{R}^{n\times r}$.

\paragraph{Support estimation:} In Section~\ref{sec04}, the operator $\bar{\cdot}$, $\hat{\cdot}$ and the partition of $\mathfrak{B}_x$ are based on the support of $x + v(x)$. In the implementation, the support of $x + v(x)$ is estimated by the following approach: if $(x + v(x))_i$ is nonzero and $|(x)_i| \geq \|v(x)\|_{\F}$, then $i$-th index of $x+v(x)$ is in the support, where $(\cdot)_i$ denotes the $i$-th entry of the argument. Therefore, the estimated support is a subset of the true support. Note that such a modification to Algorithm~\ref{alg:IRPN} does not influence its theoretical results since (i) the global convergence does not rely on the estimation of the support and (ii) if $x_k$ is sufficiently close to $x_*$ in the sense that $\|x_k - x_*\|_{\infty} < \frac{1}{2} \min\{|(x_*)_i| \mid (x_*)_i \neq 0 \}$ and $\|v(x_k)\|$ is sufficiently small in the sense that $\|v(x_k)\|_{\infty} < \frac{1}{2} \min\{|(x_*)_i| \mid (x_*)_i \neq 0 \}$, then the estimated support is equal to the true support. It follows that the local superlinear convergence analyses are still applicable.

\paragraph{Parameter setting:} The parameters of Algorithm~\ref{alg:IRPN} are set by $\alpha_{\mathrm{init}} = 1$, $\rho_1 = 0.001$, $\rho_2 = 0.5$, $\varpi_1 = 1.1$, $\varpi_2 = 0.9$, $\vartheta = 0.01$, $\gamma = 0.01$, $\tau = 100$, $\kappa = 0.1$ and $\theta = 0.5$. The parameters used in ManPG, ManPG-Ada, ManPQN, RPN-H and IAManPG are set to be the default values in the corresponding papers unless otherwise indicated, i.e.,~\cite{CMSZ2019} for ManPG and ManPG-Ada, \cite{WY2023} for ManPQN, ~\cite{SAHJV2024} for RPN-H, and~\cite{HWGV2022} for IAManPG. The values of $n$, $p$, $m$, and $\mu$ are specified later when reporting numerical results. 


For all algorithms including ManPG, ManPG-Ada, ManPQN, RPN-H, IAManPG and Algorithm~\ref{alg:IRPN}, the initial $t_0 = 1/L_f$, where $L_f$ is Lipschitz constant of $\nabla f$ in Assumption~\ref{as01}. The stopping criterion for computing the Riemannian proximal gradient direction~\eqref{eq:subforv} requires the left term in~\eqref{eq:Psilambda} to satisfy
\[
\|B_{x_k}^T \left( \prox_{t h} \bigl( x_k - t \left[\nabla f(x_k) + B_{x_k} \lambda \right] \bigr) - x_k \right)\|_{\F}^2 \leq \mathit{innertol}_k,
\]
where $\mathit{innertol}_0 = \max\left(10^{-13}, \min\left(10^{-11}, 10^{-3} * \sqrt{\mathrm{tol}} * t_0^2 \right)\right)$, $\mathrm{tol} = 10^{-8}nr$, and $\mathit{innertol}_{k} = \min\left(\max\left(10^{-30}, \|v(x_{k-1})\|_{\F}^2 * 10^{-8}\right), \mathit{innertol}_{k-1}\right)$. Unless otherwise indicated, the initial iterate is set to be the $p$ dominant singular vectors of $A$ for sparse PCA. The initial iterate of CM problem follows the same approach in~\cite{WY2023}, i.e., using a Riemannian subgradient method on a random orthonormal matrix for a few steps\footnote{It is pointed out here that RPN-CG and RPN-CGH can converge from any initial iterate on the Stiefel manifold. But ManPQN may diverge for CM problems if a random initial iterate is used. This initialization is used for complete comparisons with all the Riemannian proximal gradient-type algorithms.}. The initial iterate of CD problem is set to follow the same method in~\cite{HWGV2022}.
The termination condition used in all the tested algorithms for sparse PCA is that $\|v(x_k)\|_{\F} \leq 10^{-10}$ or the number of iterations reaches 5000. The termination condition for CM problems is that $\|v(x_k)\|_{\F} \leq 10^{-8}$ or the number of iterations reaches 3000. The termination condition for CD problem is that $\|v(x_k)\|_{\F} \leq 10^{-10}$ or the number of iterations reaches 3000. We say that two minimizers $U_1$ and $U_2$ are the same if $\min_{O^TO = I_p} \|U_1 - U_2 O\|_{\F} \leq 10^{-2}$.


The Weingarten map~\cite{AMT2013} of the Stiefel manifold is given by $\mathcal{W}_X^{\M}(W,U) = -WX^{\T}U - \frac12 X(W^{\T}U + U^{\T}W)$ and the retraction is chosen to be the polar retraction~\cite[Example~4.1.3]{AMS2008} defined as $R_X(\eta) = (X+ \eta)(\I_r + \eta^{\T}\eta)^{-1/2}$, where $X\in \M$, $\eta \in \T_X \M$. It follows from~\cite{HWGV2022} that the retraction of $\mathcal{F}_v$ is chosen to be $R_X(\eta) = vq_*^{\T}/\|v\|_2 + \mathrm{qf}(X +\eta)(\I - q_*q_*^{\T})$, where $q_* = \mathrm{qf}(X + \eta)^{\T}v/\|\mathrm{qf}(X + \eta)^{\T}v\|_2$ and $\mathrm{qf}(X + \eta)$ denotes the Q factor of the QR decomposition of $X+\eta$. The Weingarten map of $\mathcal{F}_v$ is given by $\mathcal{W}_X^{\M}(W,U)= -\frac{1}{2}W (X^{\T}U - U^{\T}W) - \frac{1}{2}X(W^{\T}U - U^{\T}W) + \ell_1 + \ell_2 ,$
where
$
\ell_1 = \left(W_v(X_v)^{\T} + X_v(W_v)^{\T}\right)/\|X_v\|_2^2 - 2(X_v)^{\T}W_vX_v(X_v)^{\T}   /\|X_v\|_2^4
$
and 
$
\ell_2 = WX^{\T}S+XW^{\T}S + XX^{\T}S
$
with
$X_v = X^{\T}v$, $W_v = W^{\T}v$ and $S = U - UX_vX_v^{\T})/\|X_v\|_2^2$.



\paragraph{Testing environment:} All experiments are performed in MATLAB R2019a on a macOS with 2.7 GHz CPU (Intel Core i7). The implementations of ManPG, ManPG-Ada, and ManPQN are from \url{https://github.com/chenshixiang/ManPG} and \url{https://github.com/QinsiWang2022/ManPQN}. The codes of RPN-CG and RPN-CGH can be found at \url{https://www.math.fsu.edu/~whuang2/papers/RPNtCG.htm} or \url{https://github.com/wutauopt/RPN-CG}.  


\subsection{RPN-H, RPN-CG, and RPN-CGH} \label{sec11}

Let RPN-CGH denote the algorithm by merging ManPG-Ada with RPN-CG. Specifically, RPN-CGH uses ManPG-Ada if $\|v(x_k)\|_{\F} > \epsilon$ and uses RPN-CG otherwise, where $\epsilon$ is the switch parameter. The parameters in the RPN-CGH are the same as those in the ManPG-Ada and RPN-CG, except for the switching parameter $\epsilon$. It is shown empirically in this section that RPN-CGH is not sensitive to the choice of the switching parameter and is more efficient than RPN-CG.


\textbf{RPN-CGH v.s. RPN-H.} Both RPN-CGH and RPN-H rely on the switching parameter $\epsilon$. However, they perform differently to the value of $\epsilon$. The sparse PCA random data are used as examples and the percentages of success run are reported in Table~\ref{tab:RPN-H_vs_RPN-CGH}. Initial iterates are generated by orthonormalizing a matrix whose entries are drawn from the standard normal distribution. A test is considered a success if it terminates due to $\|v(x_k)\|_{\F} < 10^{-10}$. RPN-H is considered a failure if the number of iterations reaches 5000 or the number of Newton steps reaches 20. Note that the number of Newton steps is usually smaller than 10 in success runs. RPN-CGH is considered a failure if the number of iterations reaches 5000. We conclude from Table~\ref{tab:RPN-H_vs_RPN-CGH} that RPN-CGH is robust to the value of $\epsilon$ in the sense that the success rate is always 100\% for all the values of $\epsilon$. However, RPN-H is sensitive to $\epsilon$. If $\epsilon$ is not sufficiently small, then RPN-H may fail to converge.

\begin{table}[ht]
    \centering
    \caption{Compare the robustness of RPN-H and RPN-CGH for different switching parameters $\epsilon$. $(n,p,\mu) = (300, 5,0.8)$. 100 random runs are used.}
    \label{tab:RPN-H_vs_RPN-CGH}
    \begin{tabular}{c|c|c|c|c|c}
    \hline
      & \multicolumn{5}{c}{$\epsilon$} \\
    \hline
      & $10^{-1}$ & $10^{-2}$ & $10^{-3}$ & $10^{-4}$ & $10^{-5}$ \\
    \hline
    RPN-H &  6\% & 15\% & 61\% & 93\% & 100\% \\
    RPN-CGH &  100\% & 100\% & 100\% & 100\% & 100\%  \\
    \hline
    \end{tabular}
\end{table}


\textbf{RPN-CGH v.s. RPN-CG.} The sparse PCA problems with random data $A$ are used to compare the efficiency of RPN-CG and RPN-CGH. The numerical results are reported in Table~\ref{tab:RPN-CG_vs_RPN-CGH}, where multiple switching parameters $\epsilon$ are used. We point out that these algorithms all converge to the same minimizer when the same random seed is used. In Table~\ref{tab:RPN-CG_vs_RPN-CGH}, RPN-CG takes the least number of iterations but the most computational time. The number of RPN-CGH takes more iterations as $\epsilon$ decreases. This is due to that ManPG-Ada generally takes more iterations than RPN-CG. Therefore, the smaller $\epsilon$ is, the more iterations ManPG-Ada takes in RPN-CGH. It follows that the number of overall iterations increases. Though RPN-CG is fast in the sense of iterations, it has extra costs in each iteration. For example, even if Algorithm~\ref{alg:tCG} terminates without entering CG iterations in the sense that it stops by Steps~\ref{a:tCG:st06} or~\ref{a:tCG:st02}, checking the conditions in Steps~\ref{a:tCG:st06} or~\ref{a:tCG:st02} takes non-negligible computational cost and the resulting search direction is still the Riemannian proximal gradient direction $v(x_k)$. This motivate us to propose the hybrid algorithm RPN-CGH, which is more efficient than RPN-CG, verified by Table~\ref{tab:RPN-CG_vs_RPN-CGH}.
In the remaining of the numerical experiments, the switch parameter $\epsilon$ of RPN-CGH is set to be $10^{-2}$.


\begin{table}[ht]
    \centering
    \caption{Compare the efficiency of RPN-CG and RPN-CGH for different switching parameters $\epsilon$. An average result of 20 random runs is reported. $(n, r,\mu) = (600, 10, 0.8)$.  The subscript $k$ indicates a scale of $10^k$.}
    \begin{tabular}{c|c|ccccc}
    \hline
     Algo & $\epsilon$   & iter & Fval & $\|v(x_k)\|_{\F}$ & CPU time & sparsity\\
    \hline
   RPN-CG & - & 394.00 & $-4.71_{1}$ & $3.30_{-11}$ & 1.21 & 0.56 \\
    \hline
RPN-CGH & $10^{-1}$ & 405.00 & $-4.71_{1}$ & $2.68_{-13}$ & 1.11 & 0.56 \\
RPN-CGH & $10^{-2}$ & 574.00 & $-4.71_{1}$ & $2.22_{-14}$ & 0.98 & 0.56 \\
RPN-CGH & $10^{-3}$ & 760.00 & $-4.71_{1}$ & $1.39_{-13}$ & 1.03 & 0.56 \\
RPN-CGH & $10^{-4}$ & 801.00 & $-4.71_{1}$ & $5.62_{-14}$ & 1.03 & 0.56 \\
RPN-CGH & $10^{-5}$ & 885.00 & $-4.71_{1}$ & $7.68_{-14}$ & 1.03 & 0.56 \\
    \hline
    \end{tabular}
    \label{tab:RPN-CG_vs_RPN-CGH}
\end{table}

\subsection{Sparse PCA}\label{sec:spca}

The numerical results from sparse PCA using random data are reported in Table~\ref{tab:spca_randn}. The parameter $m$ is set to be $50$ and parameters $n, p, \mu$ are given in the table. For a fair comparison, we take the average of 20 runs where all the algorithms converge to the same minimizer\footnote{if algorithms converge to different minimizers, then the convergence speeds of compared algorithms may differ due to the various landscapes of the cost function around the minimizers, not due to the differences of the algorithms.}.
Note that ManPG, ManPG-Ada, and ManPQN sometimes stop due to the maximum number of iterations. This is verified by the averages of $\|v(x_k)\|_{\F}$ from ManPG, ManPG-Ada, and ManPQN larger than $10^{-10}$. RPN-CG consistently takes the smallest number of iterations. Though RPN-CGH takes slightly more iterations, it is the most efficient algorithm in the sense of computational time. This result is further verified by the left two plots of Figure~\ref{fig:spca}. 

The right two plots of Figure~\ref{fig:spca} repeat the numerical experiments using synthetic data. The result again verifies that RPN-CGH is the most efficient one, and is multiple times faster than ManPG, ManPG-Ada, and ManPQN.

\begin{table}[ht]
    \centering
    \caption{\textbf{Sparse PCA.} An average result of 20 random runs for random data. Multiple values of $n$, $r$, and $\mu$ are used. The subscript $k$ indicates a scale of $10^k$. }
    \label{tab:spca_randn}
    \begin{tabular}{c|c|ccccc}
    \hline
    $(n, r, \mu)$ & Algo & iter & Fval   & $\|v(x_k)\|_{\F}$ & CPU time & sparsity  \\
    \hline
(400, 8, 0.8) & ManPG     & 3416.15 & $-2.16_{1}$ & $3.66_{-9}$ & 2.69 & 0.63 \\
(400, 8, 0.8) & ManPG-Ada & 1281.55 & $-2.16_{1}$ & $1.06_{-10}$ & 1.21 & 0.63 \\
(400, 8, 0.8) & ManPQN    & 1260.40 & $-2.16_{1}$ & $9.83_{-11}$ & 0.72 & 0.63 \\
(400, 8, 0.8) & RPN-CG    & 204.85 & $-2.16_{1}$ & $1.16_{-11}$ & 0.37 & 0.63 \\
(400, 8, 0.8) & RPN-CGH   & 294.30 & $-2.16_{1}$ & $7.13_{-12}$ & 0.33 & 0.63 \\
    \hline
(800, 8, 0.8) & ManPG     & 4232.80 & $-5.92_{1}$ & $1.84_{-7}$ & 3.56 & 0.48 \\
(800, 8, 0.8) & ManPG-Ada & 1867.05 & $-5.92_{1}$ & $2.57_{-10}$ & 1.80 & 0.48 \\
(800, 8, 0.8) & ManPQN    & 1883.80 & $-5.92_{1}$ & $1.22_{-10}$ & 1.43 & 0.48 \\
(800, 8, 0.8) & RPN-CG    & 215.05 & $-5.92_{1}$ & $1.07_{-11}$ & 0.60 & 0.48 \\
(800, 8, 0.8) & RPN-CGH   & 308.90 & $-5.92_{1}$ & $1.18_{-11}$ & 0.52 & 0.48 \\
    \hline
(400, 12, 0.8) & ManPG     & 4454.55 & $-2.82_{1}$ & $7.23_{-8}$ & 4.71 & 0.66 \\
(400, 12, 0.8) & ManPG-Ada & 1809.00 & $-2.82_{1}$ & $3.04_{-10}$ & 2.18 & 0.66 \\
(400, 12, 0.8) & ManPQN    & 1740.20 & $-2.82_{1}$ & $1.98_{-10}$ & 1.82 & 0.66 \\
(400, 12, 0.8) & RPN-CG    & 330.40 & $-2.82_{1}$ & $1.01_{-11}$ & 1.00 & 0.66 \\
(400, 12, 0.8) & RPN-CGH   & 418.35 & $-2.82_{1}$ & $1.19_{-11}$ & 0.73 & 0.66 \\
    \hline
(400, 8, 1.0) & ManPG     & 4283.25 & $-7.88$ & $1.62_{-9}$ & 3.68 & 0.77 \\
(400, 8, 1.0) & ManPG-Ada & 1131.95 & $-7.88$ & $9.74_{-11}$ & 1.12 & 0.77 \\
(400, 8, 1.0) & ManPQN    & 1553.90 & $-7.88$ & $9.86_{-11}$ & 0.88 & 0.77 \\
(400, 8, 1.0) & RPN-CG    & 254.25 & $-7.88$ & $1.00_{-11}$ & 0.48 & 0.77 \\
(400, 8, 1.0) & RPN-CGH   & 328.30 & $-7.88$ & $1.01_{-11}$ & 0.43 & 0.77 \\
    \hline
    \end{tabular}
\end{table}

\begin{figure}[ht]
\centering 
{\hspace{-0.05\textwidth} \includegraphics[width=0.52\textwidth]{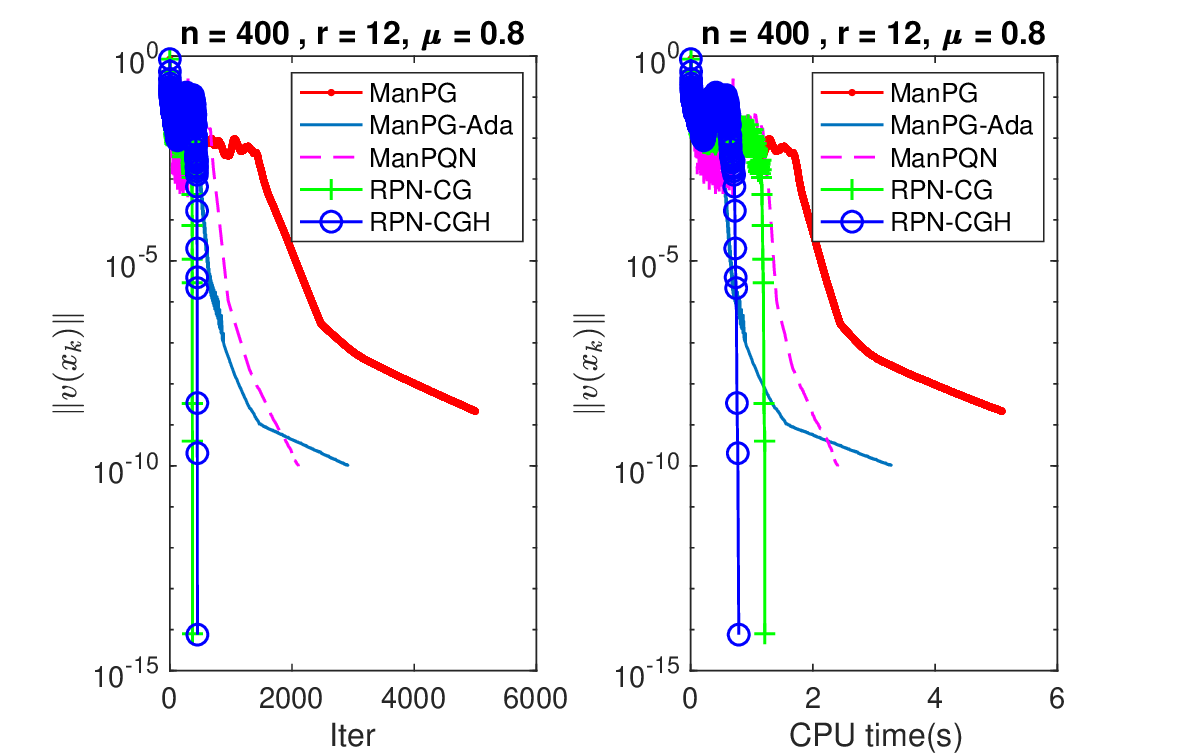}}
{\hspace{-0.05\textwidth} \includegraphics[width=0.52\textwidth]{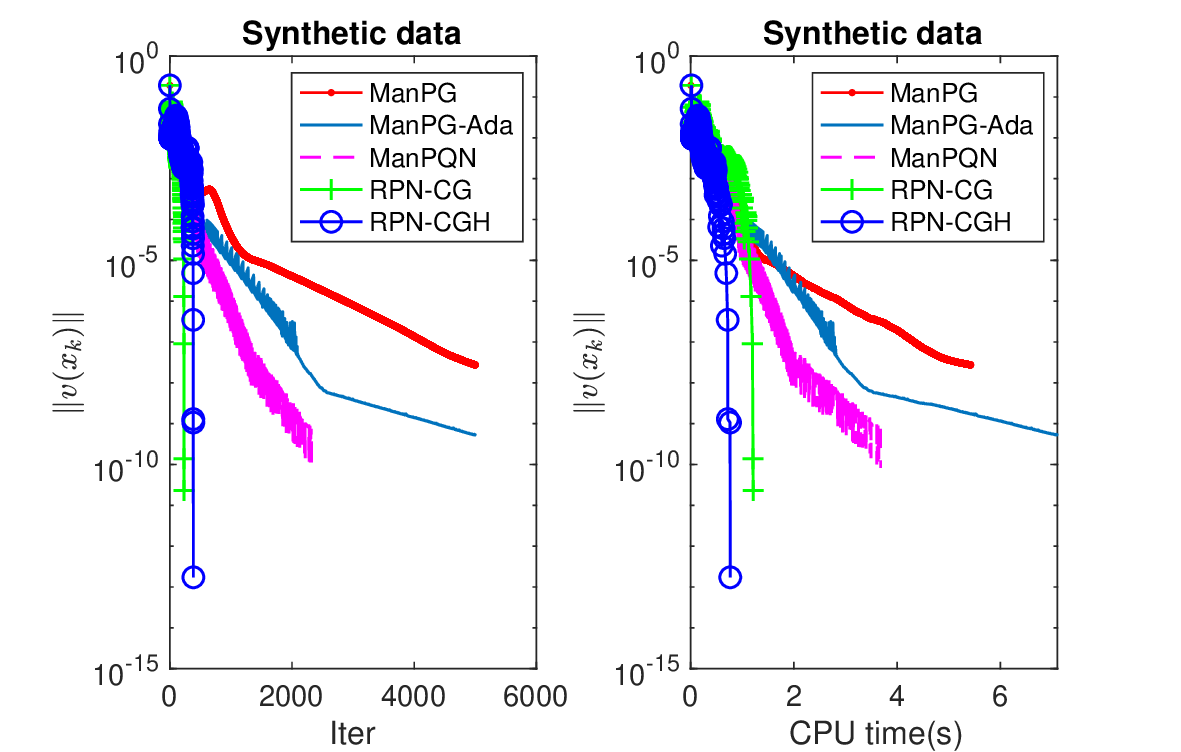}}
\caption{Sparse PCA: plots of $\|v(x_k)\|$ versus iterations and CPU times respectively. The left two plots are generated by random data and the right two plots are generated by synthetic data with $(n, p, \mu) = (4000, 5, 0.8)$ and $\epsilon = 10^{-3}$.} \label{fig:spca} 
\end{figure}

\subsection{CM Problem}\label{sec:cm}

Empirically, ManPG, ManPG-Ada, ManPQN, RPN-CG, and RPN-CGH likely converge to different minimizers for CM problems. Therefore, we report the numerical results for multiple values of $n$, $p$, and $\mu$ in Table~\ref{tab:cm_randn}, where algorithms may or may not converge to the same minimizer. 
As shown in Table~\ref{tab:cm_randn}, ManPG, ManPG-Ada, and ManPQN fail to converge in the sense of $\|v(x_k)\|_{\F} \leq 10^{-8}$ within 3000 iterations for all the 50 random runs. In contrast, RPN-CG and RPN-CGH can converge within a few hundred iterations for all the random runs. Therefore, RPN-CG and RPN-CGH are the most efficient algorithms. Note that the efficiency of RPN-CG is similar to that of RPN-CGH for CM problems. A typical run is shown in Figure~\ref{fig:cm}. Overall, for CM problems, the first-order algorithms ManPG, ManPG-Ada, and ManPQN have difficulty finding a highly accurate solution, whereas RPN-CG and RPN-CGH still work well and are more efficient.


\begin{table}[ht]
    \centering
    \caption{\textbf{CM.} An average result of 50 random runs for random data. Multiple values of $n$, $r$, and $\mu$ are used. The subscript $k$ indicates a scale of $10^k$. }
    \label{tab:cm_randn}
    \begin{tabular}{c|c|ccccc}
    \hline
    $(n, r, \mu)$ & Algo & iter & Fval   & $\|v(x_k)\|_{\F}$ & CPU time & sparsity  \\
    \hline
(256, 4, 0.1) & ManPG      & 3000.00 & $2.49$ & $4.03_{-5}$ & 0.75 & 0.85 \\
(256, 4, 0.1) & ManPG-Ada  & 3000.00 & $2.49$ & $9.49_{-5}$ & 0.88 & 0.85 \\
(256, 4, 0.1) & ManPQN     & 3000.00 & $2.49$ & $9.06_{-6}$ & 1.22 & 0.84 \\
(256, 4, 0.1) & RPN-CG     & 92.54 & $2.49$ & $2.66_{-9}$ & 0.20 & 0.86 \\  
(256, 4, 0.1) & RPN-CGH    & 92.08 & $2.49$ & $2.48_{-9}$ & 0.19 & 0.86 \\  
    \hline
(512, 4, 0.1) & ManPG      & 3000.00 & $3.29$ & $3.83_{-5}$ & 0.76 & 0.86 \\
(512, 4, 0.1) & ManPG-Ada  & 3000.00 & $3.29$ & $1.16_{-4}$ & 0.88 & 0.86 \\
(512, 4, 0.1) & ManPQN     & 3000.00 & $3.30$ & $1.44_{-6}$ & 2.98 & 0.86 \\
(512, 4, 0.1) & RPN-CG     & 147.40 & $3.29$ & $2.29_{-9}$ & 0.48 & 0.88 \\ 
(512, 4, 0.1) & RPN-CGH    & 148.30 & $3.29$ & $2.68_{-9}$ & 0.47 & 0.88 \\ 
    \hline
(256, 8, 0.1) & ManPG     & 3000.00 & $5.00$ & $1.89_{-4}$ & 4.08 & 0.81 \\
(256, 8, 0.1) & ManPG-Ada & 3000.00 & $4.99$ & $5.91_{-4}$ & 5.45 & 0.81 \\
(256, 8, 0.1) & ManPQN    & 3000.00 & $5.03$ & $4.37_{-5}$ & 1.37 & 0.80 \\
(256, 8, 0.1) & RPN-CG    & 220.96 & $4.98$ & $4.80_{-9}$ & 2.75 & 0.82 \\ 
(256, 8, 0.1) & RPN-CGH   & 200.28 & $4.98$ & $4.39_{-9}$ & 1.60 & 0.82 \\ 
    \hline
(256, 4, 0.15) & ManPG      & 3000.00 & $3.44$ & $3.28_{-5}$ & 0.65 & 0.87 \\
(256, 4, 0.15) & ManPG-Ada  & 3000.00 & $3.44$ & $7.01_{-5}$ & 0.79 & 0.87 \\
(256, 4, 0.15) & ManPQN     & 3000.00 & $3.45$ & $9.44_{-6}$ & 1.09 & 0.87 \\
(256, 4, 0.15) & RPN-CG     & 41.74 & $3.44$ & $2.16_{-9}$ & 0.07 & 0.88 \\  
(256, 4, 0.15) & RPN-CGH    & 42.14 & $3.44$ & $1.81_{-9}$ & 0.07 & 0.88 \\  
    \hline
    \end{tabular}
\end{table}

\begin{figure}[ht]
\centering 
\includegraphics[width=0.52\textwidth]{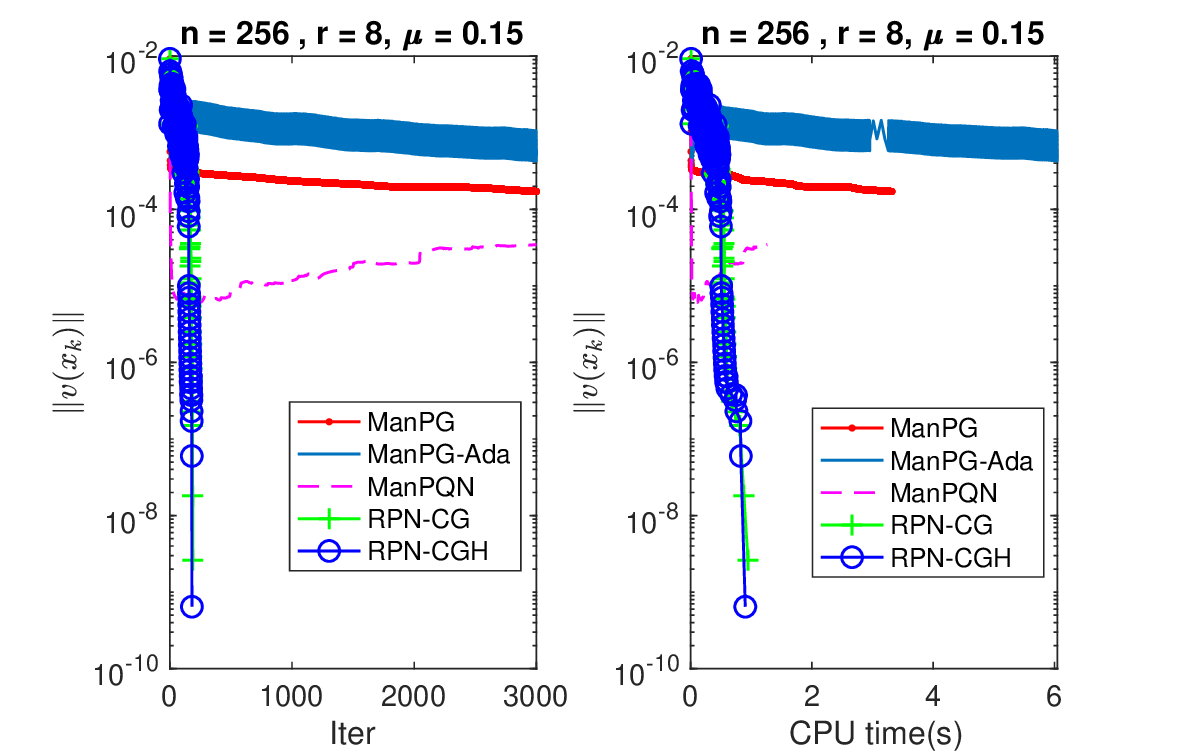}
\caption{CM: plots of $\|v(x_k)\|$ versus iterations and CPU times respectively. } \label{fig:cm} 
\end{figure}

\subsection{Community Detection}\label{subsec:cd}
In~\cite{HWGV2022}, an inexact accelerated proximal gradient method, called IAManPG, is proposed for solving~\eqref{eq:F} with $\M = \mathcal{F}_v$. We combine RPN-CG and IAmanPG and proposed a hybrid method called RPN-CGH, the switch parameter $\epsilon$ is set to be $10^{-4}$. Results are presented for solving community detection problem on synthetic LFR benchmark network~\cite{lan2008benchmark}. A software package to generate the benchmark networks is available at~\url{https://www.santofortunato.net/resources}. A typical run is shown in Figure~\ref{fig:cd}, we conclude that RPN-CG and RPN-CGH convergence faster in the sense of both computational time and the number of iterations.


\begin{figure}[ht]
\centering 
\includegraphics[width=0.52\textwidth]{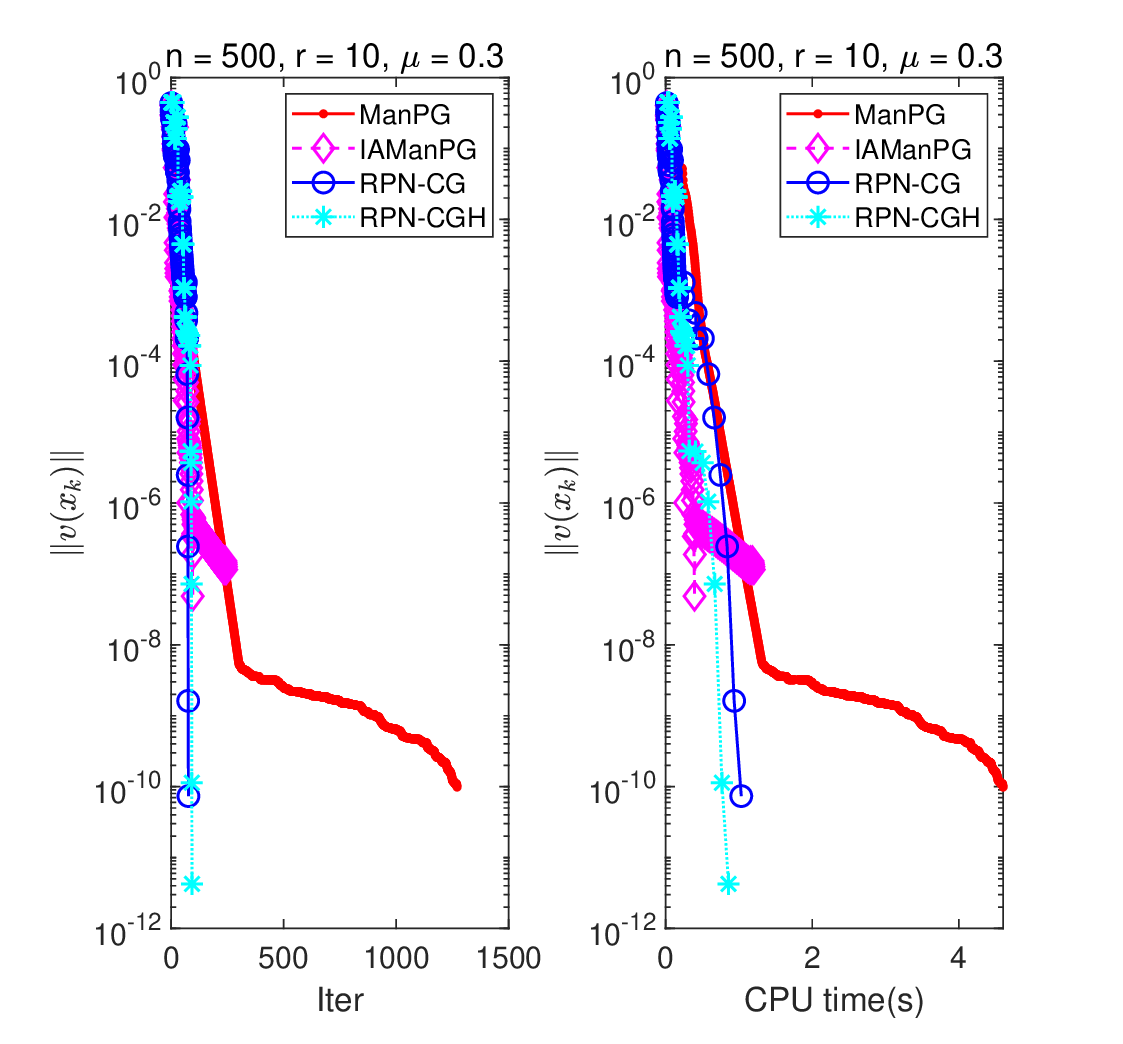}
\caption{Community Detection: plots of $\|v(x_k)\|$ versus iterations and CPU times respectively. } \label{fig:cd} 
\end{figure}

\section{Conclusion} \label{sec13}

In this paper, we proposed a Riemannian proximal Newton-CG method (RPN-CG) by integrating the ideas from the Riemannian proximal Newton method and the truncated conjugate gradient method. The global convergence and local superlinear convergence were established under reasonable assumptions. The proposed RPN-CG method overcame the difficulty that the existing hybrid approach is sensitive to the switching parameter. Numerically, we further combined ManPG-Ada with RPN-CG and proposed an RPN-CGH method. It is shown empirically that RPN-CG and RPN-CGH converge globally and superlinearly locally as desired and are much more efficient when compared to state-of-the-art methods, including ManPG, ManPG-Ada, and ManPQN.

\section*{Acknowledgements}

The authors would like to thank Pierre-Antoine Absil for discussions on optimization on Riemannian manifolds.

\bibliographystyle{plain}
\bibliography{WHlibrary}

\end{document}